\documentclass[11pt]{amsart}
\usepackage{amsopn}
\usepackage{amssymb, amscd}
\usepackage{multirow}
\usepackage{graphicx, graphics, epsfig}
\usepackage{faktor}
\usepackage{enumerate}
\usepackage{tikz}
\usepackage{pgfplots}
\usepackage{hyperref}
\usepackage{color}

\newcommand{\nc}{\newcommand}

\nc{\fg}{\mathfrak{f} } \nc{\vg}{\mathfrak{v} } \nc{\wg}{\mathfrak{w} }
\nc{\zg}{\mathfrak{z} } \nc{\ngo}{\mathfrak{n} } \nc{\kg}{\mathfrak{k} }
\nc{\mg}{\mathfrak{m} } \nc{\bg}{\mathfrak{b} } \nc{\ggo}{\mathfrak{g} } \nc{\eg}{\mathfrak{e} }
\nc{\ggob}{\overline{\mathfrak{g}} } \nc{\sog}{\mathfrak{so} }
\nc{\sug}{\mathfrak{su} } \nc{\spg}{\mathfrak{sp} } \nc{\slg}{\mathfrak{sl} }
\nc{\glg}{\mathfrak{gl} } \nc{\cg}{\mathfrak{c} } \nc{\rg}{\mathfrak{r} }
\nc{\hg}{\mathfrak{h} } \nc{\tg}{\mathfrak{t} } \nc{\ug}{\mathfrak{u} }
\nc{\dg}{\mathfrak{d} } \nc{\ag}{\mathfrak{a} } \nc{\pg}{\mathfrak{p} }
\nc{\sg}{\mathfrak{s} } \nc{\affg}{\mathfrak{aff} } \nc{\qg}{\mathfrak{q} } \nc{\lgo}{\mathfrak{l} }

\nc{\pca}{\mathcal{P}} \nc{\nca}{\mathcal{N}} \nc{\lca}{\mathcal{L}}
\nc{\oca}{\mathcal{O}} \nc{\mca}{\mathcal{M}} \nc{\tca}{\mathcal{T}}
\nc{\aca}{\mathcal{A}} \nc{\cca}{\mathcal{C}} \nc{\gca}{\mathcal{G}}
\nc{\sca}{\mathcal{S}} \nc{\hca}{\mathcal{H}} \nc{\bca}{\mathcal{B}}
\nc{\dca}{\mathcal{D}} \nc{\eca}{\mathcal{E}} \nc{\wca}{\mathcal{W}} \nc{\ica}{\mathcal{I}}

\nc{\vp}{\varphi} \nc{\ddt}{\tfrac{d}{dt}} \nc{\dsdt}{\tfrac{d^2}{dt^2}} \nc{\dds}{\frac{d}{ds}}
\nc{\dpar}{\frac{\partial}{\partial t}} \nc{\im}{\mathrm{i}}

\nc{\SO}{\mathrm{SO}} \nc{\Spe}{\mathrm{Sp}} \nc{\Sl}{\mathrm{SL}}
\nc{\SU}{\mathrm{SU}} \nc{\Or}{\mathrm{O}} \nc{\U}{\mathrm{U}} \nc{\Gl}{\mathrm{GL}}
\nc{\Se}{\mathrm{S}} \nc{\Cl}{\mathrm{Cl}} \nc{\Spin}{\mathrm{Spin}}
\nc{\Pin}{\mathrm{Pin}} \nc{\G}{\mathrm{GL}_n(\RR)} \nc{\g}{\mathfrak{gl}_n(\RR)}
\nc{\Eg}{\mathrm{E}} \nc{\Fg}{\mathrm{F}} \nc{\Gg}{\mathrm{G}}

\nc{\RR}{{\Bbb R}} \nc{\HH}{{\Bbb H}} \nc{\CC}{{\Bbb C}} \nc{\ZZ}{{\Bbb Z}}
\nc{\FF}{{\Bbb F}} \nc{\NN}{{\Bbb N}} \nc{\QQ}{{\Bbb Q}} \nc{\PP}{{\Bbb P}} \nc{\OO}{{\Bbb O}}

\nc{\vs}{\vspace{.2cm}} \nc{\vsp}{\vspace{1cm}} \nc{\ip}{\langle\cdot,\cdot\rangle}
\nc{\ipp}{(\cdot,\cdot)} \nc{\la}{\langle} \nc{\ra}{\rangle} \nc{\unm}{\tfrac{1}{2}}
\nc{\unc}{\tfrac{1}{4}} \nc{\und}{\frac{1}{16}} \nc{\no}{\vs\noindent}
\nc{\lam}{\Lambda^2(\RR^n)^*\otimes\RR^n} \nc{\tangz}{{\rm T}^{\rm Zar}}
\nc{\nor}{{\sf n}}  \nc{\mum}{/\!\!/} \nc{\kir}{/\!\!/\!\!/}
\nc{\Ri}{\tfrac{4\Ric_{\mu}}{||\mu||^2}} \nc{\ds}{\displaystyle}
\nc{\ben}{\begin{enumerate}} \nc{\een}{\end{enumerate}} \nc{\f}{\frac}
\nc{\lb}{[\cdot,\cdot]} \nc{\isn}{\tfrac{1}{||v||^2}}
\nc{\gkp}{(\ggo=\kg\oplus\pg,\ip)} \nc{\ukh}{(\ug=\kg\oplus\hg,\ip)}
\nc{\tgkp}{(\tilde{\ggo}=\kg\oplus\pg,\ip)}
\nc{\wt}{\widetilde}
\nc{\iop}{\mathtt{i}} \nc{\jop}{\mathtt{j}} 
\nc{\Hk}{H_{\kil}} \nc{\gk}{g_{\kil}}

\nc{\Hess}{\operatorname{Hess}} \nc{\ad}{\operatorname{ad}}
\nc{\Ad}{\operatorname{Ad}} \nc{\rank}{\operatorname{rk}}
\nc{\Irr}{\operatorname{Irr}} \nc{\End}{\operatorname{End}}
\nc{\Aut}{\operatorname{Aut}} \nc{\Inn}{\operatorname{Inn}}
\nc{\Der}{\operatorname{Der}} \nc{\Ker}{\operatorname{Ker}}
\nc{\Iso}{\operatorname{Iso}} \nc{\Diff}{\operatorname{Diff}}
\nc{\Lie}{\operatorname{L}} \nc{\tr}{\operatorname{tr}} \nc{\dif}{\operatorname{d}}
\nc{\sen}{\operatorname{sen}} \nc{\modu}{\operatorname{mod}}
\nc{\CRic}{\operatorname{PP}} \nc{\Cric}{\operatorname{P}} \nc{\Ricci}{\operatorname{Ric}}
\nc{\sym}{\operatorname{sym}} \nc{\herm}{\operatorname{herm}} \nc{\symac}{\operatorname{sym^{ac}}}
\nc{\symc}{\operatorname{sym^{c}}} \nc{\scalar}{\operatorname{Sc}}
\nc{\grad}{\operatorname{grad}} \nc{\ricci}{\operatorname{Rc}} \nc{\kil}{\operatorname{B}} \nc{\cas}{\operatorname{C}} \nc{\lic}{\operatorname{L}}
\nc{\Nor}{\operatorname{Norm}}  \nc{\ricc}{\operatorname{Rc^{c}}}
\nc{\Ricc}{\operatorname{Ric^{c}}} \nc{\ricac}{\operatorname{Rc^{ac}}}
\nc{\Ricac}{\operatorname{Ric^{ac}}} \nc{\Riem}{\operatorname{Rm}} \nc{\Sec}{\operatorname{Sec}}
\nc{\riccig}{\operatorname{ric^{\gamma}}} \nc{\mm}{\operatorname{m}} \nc{\Mm}{\operatorname{M}}
\nc{\Le}{\operatorname{L}} \nc{\tang}{\operatorname{T}}
\nc{\level}{\operatorname{level}} \nc{\rad}{\operatorname{r}}
\nc{\abel}{\operatorname{ab}} \nc{\CH}{\operatorname{CH}} \nc{\Cone}{{\mathcal C}} \nc{\CCone}{\operatorname{CC}} \nc{\CP}{{\mathcal P}}
\nc{\mcc}{\operatorname{mcc}} \nc{\Adj}{\operatorname{Adj}}
\nc{\Order}{\operatorname{O}}  \nc{\inj}{\operatorname{inj}} \nc{\proy}{\operatorname{pr}}
\nc{\vol}{\operatorname{vol}} \nc{\Diag}{\operatorname{Dg}} \nc{\Diagg}{\operatorname{Diag}}
\nc{\Spec}{\operatorname{Spec}} \nc{\Ima}{\operatorname{Im}} \nc{\Rea}{\operatorname{Re}}
\nc{\spann}{\operatorname{span}} \nc{\Aff}{\operatorname{Aff}} \nc{\E}{\operatorname{E}} \nc{\id}{\operatorname{id}} \nc{\dete}{\operatorname{det}} \nc{\Crit}{\operatorname{Crit}} \nc{\val}{\operatorname{val}}

\theoremstyle{plain}
\newtheorem{theorem}{Theorem}[section]
\newtheorem{proposition}[theorem]{Proposition}
\newtheorem{corollary}[theorem]{Corollary}
\newtheorem{lemma}[theorem]{Lemma}

\theoremstyle{definition}
\newtheorem{definition}[theorem]{Definition}

\newtheorem{assum}[theorem]{Assumption}

\theoremstyle{remark}
\newtheorem{remark}[theorem]{Remark}
\newtheorem{example}[theorem]{Example}

\title{Einstein metrics on aligned homogeneous spaces with two factors}

\author{Jorge Lauret}  
\author{Cynthia Will}

\address{FaMAF, Universidad Nacional de C\'ordoba and CIEM, CONICET (Argentina)}
\email{jorgelauret@unc.edu.ar} 
\email{cynthia.will@unc.edu.ar}

\thanks{This research was partially supported by three grants from, respectively,  CONICET, Univ. Nac. de C\'ordoba and Foncyt (Argentina)}

\date{\today}

\begin{document}

\maketitle

\begin{abstract}
Given two homogeneous spaces of the form $G_1/K$ and $G_2/K$, where $G_1$ and $G_2$ are compact simple Lie groups, we study the existence problem for $G_1\times G_2$-invariant Einstein metrics on the homogeneous space $M=G_1\times G_2/K$.  For the large subclass $\cca$ of spaces having three pairwise inequivalent  isotropy irreducible summands ($12$ infinite families and $70$ sporadic examples), we obtain that existence is equivalent to the existence of a real root for certain quartic polynomial depending on the dimensions and two Killing constants, which allows a full classification and the possibility to weigh the existence and non-existence pieces of $\cca$.     
\end{abstract}

\tableofcontents

\section{Introduction}\label{intro}

A major open problem in homogeneous Riemannian geometry asks which compact homogeneous spaces $M=G/K$ admit a $G$-invariant Einstein metric.  The necessary and/or sufficient conditions may be in terms of algebraic or Lie theoretical properties of $G$, $K$ and the embedding $K\subset G$, as well as of topological properties of $M$.  However, it is not actually clear what would be a satisfactory answer, if any.  Only three main general sufficient conditions for existence are known, which were obtained by Wang-Ziller \cite{WngZll} and B\"ohm-Wang-Ziller \cite{BhmWngZll}, B\"ohm \cite{Bhm} and Graev \cite{Grv} in terms of, respectively, a graph,  a simplicial complex and a compact semialgebraic set (nerve), all attached to the space of intermediate subalgebras $\kg\subset\hg\subset\ggo$ and their flags (see \cite{BhmKrr2} for a recent exposition on all these deep results).  

In this light, as proposed in \cite{BhmKrr2}, given a large class $\cca$ of homogeneous spaces such that the above sufficient conditions do not hold for any member of $\cca$, one may try to find a necessary and sufficient condition for the existence of an invariant Einstein metric and ponder the existence and non-existence parts of $\cca$.  What is more likely?  This was done for the class of all homogeneous spaces with only two irreducible isotropy representation components in \cite[Theorem 3.1]{WngZll}: existence is equivalent to the existence of a real root for a quadratic polynomial whose coefficients depend on the dimensions of the irreducible components, one Killing constant and two Casimir constants.  A complete classification was obtained in \cite{DckKrr}, providing several non-existence examples as well as existence cases which do not satisfy any of the known sufficient conditions.  Existence is highly likely when $G$ is classical but it almost ties with non-existence for $G$ exceptional.  Two other classes, denoted by $\nca_<$ and $\nca_>$ were studied from this point of view in \cite{Bhm2, BhmKrr2}, though a ponderation of the existence part is still missing.  

In this paper, we consider compact semisimple Lie groups with two simple factors $G=G_1\times G_2$ and homogeneous spaces $M=G/K$ such that $K$ projects non-trivially on both factors.  It is well known that the third Betti number $b_3(M)$ is therefore $\leq 1$ (see \cite{H3}).  We are interested in the case when $b_3(M)=1$, so called {\it aligned} homogeneous spaces (see \cite{H3, BRF}).  For instance, the space $\SU(m)\times\SU(m)/\U(k)_{p,q}$ has $b_3=1$ if and only if $p=q$ (see Example \ref{mmkpq}).  Algebraically, the aligned condition is equivalent to 
$$
\kil_\ggo(Z,Z)=c_1\kil_{\ggo_1}(Z_1,Z_1)=c_2\kil_{\ggo_2}(Z_2,Z_2), \qquad\forall Z=(Z_1,Z_2)\in\kg\subset\ggo=\ggo_1\oplus\ggo_2, 
$$
for unique positive numbers $c_1,c_2$ such that $\tfrac{1}{c_1}+\tfrac{1}{c_2}=1$ (see Definition \ref{alig-def-2} for an alternative equivalent algebraic condition in terms of the Killing constants $\kil_{\pi_i(\kg_j)}=a_{ij}\kil_{\ggo_i}$ of the different simple factors of $\kg$ supporting the name aligned).  Note that $G/K$ is aligned as soon as $K$ is simple or one-dimensional and that $\kg$ is automatically isomorphic to its projection on $\ggo_i$ for $i=1,2$.  

On each aligned space $M^n=G_1\times G_2/K$, a $3$-parameter family of $G$-invariant metrics $g=(x_1,x_2,x_3)$ can be defined in the usual way by using the $\kil_\ggo$-orthogonal reductive decomposition $\ggo=\kg\oplus\pg$ and the $\kil_\ggo$-orthogonal $\Ad(K)$-invariant decomposition 
\begin{equation*}
\pg=\pg_1\oplus\pg_2\oplus\pg_3, \qquad \mbox{where}\quad \pg_3 = \left\{ \left(Z_1,-\tfrac{1}{c_1-1}Z_2\right):Z\in\kg\right\},
\end{equation*}
and $\pg_i$, $i=1,2$, is identified with the subspace of $\ggo_i$ coming from the $\kil_{\ggo_i}$-orthogonal reductive decomposition $\ggo_i=\pi_i(\kg)\oplus\pg_i$ of the homogeneous space $M_i^{n_i}:=G_i/\pi_i(K)$.  Note that $n=n_1+n_2+d$, where $d:=\dim{K}$.  The Ricci curvature of these metrics was computed in \cite{BRF}, they have $3+t$ Ricci eigenvalues ($2+t$ if $K$ is semisimple), where $t$ is the number of simple factors of $K$ (see Proposition \ref{ricggals2}).  Note that in general the space $\mca^G$ of all $G$-invariant metrics can be much larger.  

Our main result concerns the existence problem for Einstein metrics of the form $g=(x_1,x_2,x_3)$.  The case when $G_1=G_2$ and $K$ is diagonally embedded, i.e., $M=H\times H/\Delta K$ for some homogeneous space $H/K$, has already been studied in \cite{HHK}: existence holds if and only if the Casimir operator of the isotropy representation of $H/K$ satisfies that $\cas_\chi=\kappa I_\qg$ for some $\kappa\in\RR$ (i.e., the standard metric on $H/K$ is Einstein), $\kil_\kg=a\kil_\hg|_\kg$ for some $a\in\RR$ and      
$(2\kappa+1)^2 \geq 8a(1-a+\kappa)$.  This inequality holds for most of the spaces satisfying the first two structural conditions,  which consist of $17$ infinite families and $50$ sporadic examples.     

\begin{theorem}\label{s2E-intro}
If an aligned homogeneous space $M=G_1\times G_2/K$ admits an Einstein metric of the form $g=(x_1,x_2,x_3)$, then, for $i=1,2$, the Casimir operator of $G_i/\pi_i(K)$ is given by $\cas_{\chi_i}=\kappa_iI_{\pg_i}$ for some $\kappa_i>0$ (i.e., the standard metric on $G_i/\pi_i(K)$ is Einstein) and 
\begin{enumerate}[{\rm (i)}]
\item either $K$ is abelian and there exists exactly one Einstein metric up to scaling,   

\item or $K$ is semisimple and $\kil_{\pi_i(\kg)}=a_i\kil_{\ggo_i}|_{\pi_i(\kg)}$ for some $0<a_i\leq 1$, $i=1,2$ (e.g., $K$ simple).  In that case, the existence is equivalent to the existence of a real root for certain quartic polynomial $p$ whose coefficients depend on $n_1$, $n_2$, $d$, $a_1$, $a_2$ (here $c_i=\tfrac{a_1+a_2}{a_i}$, $\kappa_i=\tfrac{d(1-a_i)}{n_i}$).  
\end{enumerate}
Moreover, the Einstein metric $g$ is always unstable as a critical point of the scalar curvature functional (see Figure \ref{fig}).
\end{theorem}

The class of homogeneous spaces involved in the above theorem is quite large and can be described using the classification of isotropy irreducible spaces obtained by Wolf (see \cite{Bss}) and the classification given in \cite{WngZll2} by Wang and Ziller (see also \cite{stab}): 
\begin{enumerate}[{\small $\bullet$}]
\item $K$ abelian: $1$ infinite family and $7$ sporadic examples.  Here $K$ is a maximal torus of both $G_1$ and $G_2$ (see \S\ref{Kab-sec}).   

\item $K$ simple: $12$ infinite families and $99$ sporadic examples (see \S\ref{Kss-sec}).  

\item $K$ semisimple, non-simple: $2$ infinite families and $24$ sporadic examples.  
\end{enumerate}

The quartic polynomial mentioned in part (ii) of Theorem \ref{s2E-intro} depends only on $n_1$, $n_2$, $d$, $a_1$, $a_2$ but unfortunately, in a very complicated way (see \eqref{pE}), making of the existence problem a really tricky task for $K$ semisimple.    

In \S\ref{mf-sec}, we focus on the class $\cca$ of all aligned spaces $M=G_1\times G_2/K$ such that 
$$
\mca^G=\{ g=(x_1,x_2,x_3):x_i>0\}, 
$$
that is, $G_1/\pi_1(K)$ and $G_2/\pi_2(K)$ are two different isotropy irreducible spaces and $K$ is simple.  The existence of a $G$-invariant Einstein metric on a space in $\cca$ is therefore equivalent to the existence of a real root for the quartic polynomial $p$ (see Theorem \ref{s2E-intro}, (ii)).  Such existence cannot follow from global reasons since there are only three intermediate subalgebras, one of which is contained in the other two, so the graph is always connected and the B\"ohm's simplicial complex and Graev's nerve are both contractible (see \cite{BhmKrr2}).  

The class $\cca$ is still huge, it consists of $12$ infinite families and $70$ sporadic examples (see Table \ref{ii-simple}).  With the help of Maple, we compute the discriminant and other two invariants of the quartic polynomial $p$ in order to solve the existence problem, obtaining the following results: 
\begin{enumerate}[{\small $\bullet$}]
\item The $12$ families are given in Table \ref{flies}.  Existence is much more likely, there are only $3$ non-existence infinite families.   

\item All the spaces such that $G_1/\pi_1(K)$ and $G_2/\pi_2(K)$ are both irreducible symmetric spaces are listed in Table \ref{sym} ($1$ family and $5$ sporadic examples).  There exists an Einstein metric only on one of them in this small subclass.  

\item In Tables \ref{spo} and \ref{spo2}, the remaining $65$ sporadic examples are given.  An invariant Einstein metric exists on exactly $51$ of these spaces.    
\end{enumerate}

Summarizing, among the $70$ sporadic spaces in $\cca$, existence holds exactly for $52$ of them and for $9$ of the $12$ families, so the existence rate on the class $\cca$ is aproximately $75\%$.  

In all the existence cases there are exactly two invariant Einstein metrics.  We note that our exploration provides several new examples of homogeneous spaces with three isotropy irreducible summands which do not admit invariant Einstein metrics.

\vs \noindent {\it Acknowledgements.}  
We are very grateful with Christoph B\"ohm for many helpful comments and suggestions on a first version of this paper.

\section{Aligned homogeneous spaces}\label{preli3}

Homogeneous spaces with the richest third cohomology (other than Lie groups), i.e., the third Betti number satisfies that $b_3(G/K)=s-1$ if $G$ has $s$ simple factors, are called {\it aligned} homogeneous spaces.  We overview in this section the case when $s=2$, which are the homogeneous spaces studied in this paper regarding the existence of invariant Einstein metrics.  See \cite{H3, BRF} for more complete treatments.  

\subsection{Definition}\label{def-sec} 
Given a compact and connected differentiable manifold $M^n$ which is homogeneous, we fix an almost-effective transitive action of a compact connected Lie group $G$ on $M$.  The $G$-action determines a presentation $M=G/K$ of $M$ as a homogeneous space, where $K\subset G$ is the isotropy subgroup at some point $o\in M$. 

We assume that $G$ is semisimple with two simple factors and we consider the decompositions for the corresponding Lie algebras, 
\begin{equation}\label{decs}
\ggo=\ggo_1\oplus\ggo_2, \qquad \kg=\kg_0\oplus\kg_1\oplus\dots\oplus\kg_t, 
\end{equation}
where the $\ggo_i$'s and $\kg_j$'s are simple ideals of $\ggo$ and $\kg$, respectively, and $\kg_0$ is the center of $\kg$.  If $\pi_i:\ggo\rightarrow\ggo_i$ is the usual projection, then we set $Z_i:=\pi_i(Z)$ for any $Z\in\ggo$, so $Z=(Z_1,Z_2)$.  

\begin{remark}\label{GiKj}
Up to finite cover, we have that 
$$
M=G_1\times G_2/K_0\times K_1\times\dots\times K_t,
$$ 
where each $G_i$ and $K_j$ is a Lie group with Lie algebra $\ggo_i$ and $\kg_j$, respectively.  
\end{remark}

The Killing form of a Lie algebra $\hg$ will always be denoted by $\kil_\hg$.  We consider the {\it Killing constants}, defined by 
$$
\kil_{\pi_i(\kg_j)} = a_{ij}\kil_{\ggo_i}|_{\pi_i(\kg_j)}, \qquad i=1,2, \quad j=0,1,\dots,t.  
$$ 
Note that $0\leq a_{ij}\leq 1$, $a_{ij}=0$ if and only if $j=0$ or $\pi_i(\kg_j)=0$, and $a_{ij}=1$ if and only if $\pi_i(\kg_j)=\ggo_i$ (see \cite{DtrZll} for a deep study of these constants).  

\begin{definition}\label{alig-def-2}
A homogeneous space $G/K$ as above with $K$ semisimple (i.e., $\kg_0=0$) is said to be {\it aligned} if $\pi_i(\kg_j)\ne 0$ (i.e., $a_{ij}>0$) for all $i,j$ and the vectors of $\RR^2$ given by  
$$
(a_{1j},a_{2j}), \qquad j=1,\dots,t,
$$ 
are all collinear, i.e., there exist numbers $c_1,c_2>0$ with $\frac{1}{c_1}+\frac{1}{c_2}=1$ such that
$$
(a_{1j},a_{2j}) = \lambda_j(c_1,\dots,c_2) \quad\mbox{for some}\quad \lambda_j>0, \quad\forall j=1,\dots,t \quad \mbox{(i.e., $a_{ij}=\lambda_jc_i$)}.
$$
In the case when $\kg_0\ne 0$, $G/K$ is called {\it aligned} if in addition to the above conditions,     
\begin{equation}\label{al3} 
\kil_{\ggo_i}(Z_i,W_i) = \tfrac{1}{c_i}\kil_\ggo(Z,W), \qquad\forall Z,W\in\kg_0, \quad i=1,2.
\end{equation}   
Since $a_{i0}=0$, we set $\lambda_0:=0$.
\end{definition}

In other words, the ideals $\kg_j$'s are uniformly embedded in each $\ggo_i$ in some sense.  Note that $G/K$ is automatically aligned if $\kg$ is simple or one-dimensional, provided that $\pi_i(\kg)\ne 0$ for $i=1,2$.  Thus any pair $G_1/K$, $G_2/K$ with $K$ simple determines an aligned space, which in particular shows that this is a wild class in some sense, it is just too large, a classification in the usual sense is out of reach.   

The following properties of an aligned homogeneous space $G/K$ easily follow (see \cite{H3}): 
\begin{enumerate}[{\small $\bullet$}]
\item $\pi_i(\kg)\simeq\kg$ for $i=1,2$.  

\item For any $Z,W\in\kg$, 
\begin{equation}\label{al2} 
\kil_{\ggo_i}(Z_i,W_i) = \tfrac{1}{c_i}\kil_\ggo(Z,W), \qquad  i=1,2.  
\end{equation} 
The existence of $c_1,c_2>0$ such that \eqref{al2} holds is an alternative definition of the notion of aligned.  

\item The Killing form of $\kg_j$ is given by
\begin{equation}\label{al1}
\kil_{\kg_j}=\lambda_j\kil_{\ggo}|_{\kg_j}, \qquad\forall j=1,\dots,t.  
\end{equation}
\end{enumerate}
Under the assumption that $\pi_i(\kg)\ne 0$ for $i=1,2$, any homogeneous space $G/K$ of a semisimple $G$ with two simple factors has $b_3(G/K)\leq 1$, where equality holds if and only if $G/K$ is aligned, which is in turn equivalent to the existence of an inner product $\ip$ on $\kg$ such that $Q|_{\kg\times\kg}$ coincides with $\ip$ up to scaling for any bi-invariant symmetric bilinear form $Q$ on $\ggo$ (see \cite[Proposition 4.10]{H3}).

\subsection{Examples}\label{exa-sec}
We now list some  examples and constructions of aligned homogeneous spaces with two factors, as defined in the above section.  

\begin{example}\label{dim5}
The lowest dimensional examples are 
$$
M^5=\SU(2)\times\SU(2)/S^1_{p,q}, \qquad p,q\in\NN, 
$$
where $K=S^1_{p,q}$ is embedded with slope $(p,q)$, i.e., $\kg=\RR(pZ,qZ)$, 
$Z:=\left[\begin{smallmatrix} \im&0\\ 0&-\im\end{smallmatrix}\right]$.  Using that $\kil_{\sug(2)}(X,Y)=4\tr{XY}$, we obtain from \eqref{al3} that $c_1=\tfrac{p^2+q^2}{p^2}$ and $c_2=\tfrac{p^2+q^2}{q^2}$.  Note that $c_1=c_2=2$ if and only if $p=q$.  All these manifolds are diffeomorphic to $S^2\times S^3$, but two of them are equivariantly diffeomorphic if and only if $p/q=p'/q'$.  
\end{example}

\begin{example}\label{mmkpq}
Consider the homogeneous spaces 
$$
M_{p,q}=\SU(m)\times\SU(m)/\U(k)_{p,q}, \qquad k<m, 
$$
where either $p,q\in\NN$ are coprime or $p=q=1$ and the center of $K=\U(k)_{p,q}$ is embedded with slope $(p,q)$, say, 
$$
\kg=\Delta\sug(k)\oplus\RR(pZ,qZ), \qquad\mbox{where}\quad
Z:=\left[\begin{smallmatrix} (m-k)\im I_k&0\\ 0&-k\im I_{m-k}\end{smallmatrix}\right]\in\sug(m).
$$  
Since $a_{11}=a_{12}$, it follows from Definition \ref{alig-def-2} that we must have $c_1=c_2=2$, which implies that this space is aligned if and only if $\kil_{\sug(m)}(pZ,pZ)=\kil_{\sug(m)}(qZ,qZ)$, that is, $p=q=1$.  Remarkably, when $k=m-1$, it is proved in \cite{BhmKrr} that $M_{p,q}$ admits an invariant Einstein metric if and only if $p=q=1$.  The authors notice that the homology group $H_4(M_{p,q},\ZZ)$ is torsion-free if and only if $p=q=1$, relating the existence to a topological property.  We deduce from our viewpoint that there is an additional topological characterization for the existence of invariant Einstein metrics; indeed, $b_3(M_{p,q})\leq 1$ for all $p,q$ and equality holds if and only if $p=q=1$.
\end{example}

\begin{example}\label{GLO}
The following case was studied in \cite{HHK}.  If $\ggo_1=\ggo_2=\hg$ and $\pi_1=\pi_2$, i.e., $G=H\times H$, $H$ simple and $K\subset H$ a subgroup, then $G/\Delta K$ is aligned with 
$$
c_1=c_2=2, \qquad  \lambda_1=\tfrac{a_1}{2},\dots,\lambda_t=\tfrac{a_t}{2}, 
$$
where $\kil_{\kg_j}=a_j\kil_\hg|_{\kg_j}$ for each simple factor $\kg_j$ of $\kg$.  It is easy to see that $M=G/\Delta K$ is diffeomorphic to $(H/K)\times H$.  In the particular case when $K=H$, $M$ is a symmetric space.     
\end{example}

\begin{example}\label{kill-exa} 
Given two compact homogeneous spaces $G_1/H_1$ and $G_2/H_2$ such that $G_i$ is simple, $H_i\simeq K$ and $\kil_{\hg_i}=a_i\kil_{\ggo_i}|_{\hg_i}$, $a_i>0$ (e.g.\ if $K$ is simple, see \cite[pp.35]{DtrZll}) for $i=1,2$, we consider $M=G/\Delta K$, where $G:=G_1\times G_2$, $\Delta K:=\{ (\theta_1(k),\theta_2(k)):k\in K\}$ and $\theta_i:K\rightarrow H_i$ a Lie group isomorphism.  Note that $K$ is necessarily semisimple.  It is easy to see that $M=G/\Delta K$ is an aligned homogeneous space with 
$$
c_1=a_1\sum_{r=1}^2\tfrac{1}{a_r}, \quad  c_2=a_2\sum_{r=1}^2\tfrac{1}{a_r}, 
\qquad \lambda_1=\dots=\lambda_t=\left(\sum_{r=1}^2\tfrac{1}{a_r}\right)^{-1},
$$
and also that any aligned homogeneous space with $K$ semisimple and $\lambda_1=\dots=\lambda_t$ can be constructed in this way.  Note that if $G_1=G_2$, then $a_1=a_2$, $c_1=c_2=2$ and so we recover Example \ref{GLO}.  If $a_1\leq a_2$ then 
$$
1<c_1=\tfrac{a_1+a_2}{a_2}\leq 2\leq c_2=\tfrac{a_1+a_2}{a_1}, \qquad \lambda_j=\tfrac{a_1a_2}{a_1+a_2}, \quad\forall j.
$$    
\end{example}

\begin{example}\label{ex3}
Consider $M=\SU(n_1)\times\SU(n_2)/\SU(k_1)\times\dots\times\SU(k_t)$, where $k_1+\dots+k_t< n_i$ and the standard block diagonal embedding are taken.  It follows from \cite[pp.37]{DtrZll} that $a_{ij}=\tfrac{k_j}{n_i}$, which implies that this space is aligned with 
$$
c_1=\tfrac{n_1+n_2}{n_1}, \qquad c_2=\tfrac{n_1+n_2}{n_2}, \qquad \lambda_j=\tfrac{k_j}{n_1+\dots+n_s}.  
$$
These aligned spaces are therefore different from those provided by Examples \ref{GLO} and \ref{kill-exa}.  
\end{example}

\begin{example}\label{ex4}
 It follows from \cite[pp.38]{DtrZll} and Example \ref{kill-exa} that the spaces (with the standard embeddings)
 $$
 M^{45}=\SU(6)\times\SO(8)/\SU(3)\times\Spe(2), \quad M^{106}=\SO(14)\times \Eg_6/\SU(6)\times\SO(8),
 $$
 are both aligned with $c_1=c_2=2$ and $\lambda_1=\lambda_2=\unc$, since all the Killing constants involved are equal to $\unm$ (the embedding of $\SU(3)$ in $\SO(8)$ considered is $(\CC^3\oplus\overline{\CC^3})_\RR\oplus\RR^2$).  Note that the same holds with $\lambda_1=\unc$ if one considers only one of the simple factors of $K$.  
\end{example}

We note that an aligned space has $c_1=c_2=2$ if and only if $a_{1j}=a_{2j}=:a_j$ for any $j=1,\dots,t$ (unless $K$ is abelian).  In that case, $\lambda_j=\tfrac{a_j}{2}$ for all $j$.   

\begin{example}\label{so} 
Given any compact homogeneous space $G_2/K$ with $G_2$ simple and $K$ semisimple, we consider the homogeneous space $\SO(d)/K$, where $d=\dim{K}$ and the embedding is determined by the adjoint representation of $K$ on $\RR^d=\kg$ (which is isotropy irreducible if $K$ is simple, see \cite[7.49]{Bss}).  According to the construction given in Example \ref{kill-exa}, if we assume that $\kil_\kg=a_2\kil_{\ggo_2}$, then $M^n=\SO(d)\times G_2/\Delta K$, $n=\tfrac{d(d-1)}{2}+n_2$, is an aligned homogeneous space with 
$$
c_1=\tfrac{(d-2)a_2+1}{(d-2)a_2}, \qquad c_2=(d-2)a_2+1, \qquad \lambda_1=\dots=\lambda_t=\tfrac{a_2}{(d-2)a_2+1}.
$$
We are using here that $a_1=\tfrac{1}{d-2}$ (see \cite[Section 7]{stab}).
\end{example}

\subsection{Reductive decomposition}\label{red-sec}
Let $\mca^G$ denote the finite-dimensional manifold of all $G$-invariant Riemannian metrics on a compact homogeneous space $M=G/K$.  For any reductive decomposition $\ggo=\kg\oplus\pg$ (i.e., $\Ad(K)\pg\subset\pg$), giving rise to the usual identification $T_oM\equiv\pg$, we identify any $g\in\mca^G$ with the corresponding $\Ad(K)$-invariant inner product on $\pg$, also denoted by $g$.  

We assume from now on that $M=G/K$ is an aligned homogeneous space with two factors as in Definition \ref{alig-def-2}.  We consider the $\kil_\ggo$-orthogonal reductive decomposition $\ggo=\kg\oplus\pg$ and the $G$-invariant metric $\gk$ defined by $\gk=-\kil_\ggo|_\pg$, so called {\it standard}, as a background metric, and the $\gk$-orthogonal $\Ad(K)$-invariant decomposition 
\begin{equation*}
\pg=\pg_1\oplus\pg_2\oplus\pg_3, \qquad \mbox{where}\quad \pg_3 = \left\{ \left(Z_1,-\tfrac{1}{c_1-1}Z_2\right):Z\in\kg\right\},
\end{equation*}
(recall that $c_2=\tfrac{c_1}{c_1-1}$).  Here each $\pg_i$, $i=1,2$, is identified with the subspace of $\ggo_i$ coming from the $\kil_{\ggo_i}$-orthogonal reductive decomposition 
$$
\ggo_i=\pi_i(\kg)\oplus\pg_i, \qquad i=1,2,
$$ 
of the homogeneous space $M_i:=G_i/\pi_i(K)$.  In this way, as $\Ad(K)$-representations, $\pg_i$ is equivalent to the isotropy representation of the homogeneous space $G_i/\pi_i(K)$ for each $i=1,2$ and $\pg_3$ is equivalent to the adjoint representation of $\kg$.  We note that 
$
\pi_1(\kg)\oplus\pi_2(\kg) = \pg_3\oplus\kg  
$
is a Lie subalgebra of $\ggo$, which is abelian if and only if $\kg$ is abelian.  It is therefore easy to check that 
\begin{align}
&[\pg_1,\pg_1]\subset\pg_1+\pg_3+\kg, \label{pipj0s2}\\  
&[\pg_2,\pg_2]\subset\pg_2+\pg_3+\kg, \label{pipj1s2}\\  
&[\pg_3,\pg_1]\subset\pg_1, \qquad [\pg_3,\pg_2]\subset\pg_2, \label{pipj2s2}, \\ 
&[\pg_3,\pg_3] \subset\pg_3+\kg. \label{pipj3s2}
\end{align}

The subspace $\pg_3$ in turn admits an $\Ad(K)$-invariant decomposition 
\begin{equation}\label{pjdec}
\pg_3=\pg_3^0\oplus\pg_3^1\oplus\dots\oplus\pg_3^t, 
\end{equation}
which is also $\gk$-orthogonal, and for any $l=0,\dots,t$, the subspace $\pg_3^l$ is equivalent to the adjoint representation $\kg_l$ as an $\Ad(K)$-representation (see \cite[Proposition 5.1]{H3}); in particular, $\pg_3^l$ is $\Ad(K)$-irreducible for any $1\leq l$ and they are pairwise inequivalent. 

We focus in this paper on the $G$-invariant metrics of the form
\begin{equation*}
g=x_1\gk|_{\pg_1}+x_2\gk|_{\pg_2}+x_3\gk|_{\pg_3}, \qquad x_1,x_2,x_3>0,
\end{equation*} 
which will be denoted by 
\begin{equation}\label{metg}
g=(x_1,x_2,x_3).  
\end{equation}   
The following notation will be used throughout the paper:
\begin{align*}
&d:=\dim{K}, \qquad d_l:=\dim{\kg_l}, \quad l=0,\dots,t, \qquad \mbox{so}\quad  d=d_0+d_1+\dots+d_t, \\ 
&n_i:=\dim{\pg_i}=\dim{G_i}-d, \quad i=1,2, \quad n:=\dim{M}=n_1+n_2+d.
\end{align*}

\subsection{Ricci curvature}\label{ric-sec}
We consider, for $i=1,2$, the homogeneous space $M_i=G_i/\pi_i(K)$ (see Remark \ref{GiKj}) with $\kil_{\ggo_i}$-orthogonal reductive decomposition $\ggo_i=\pi_i(\kg)\oplus\pg_i$ endowed with its standard metric, which will be denoted by $\gk^i$.  According to \cite[Proposition (1.91)]{WngZll2} (see also \cite[(5)]{HHK} and \cite[(6)]{stab}), 
\begin{equation}\label{MgBi}
\Ricci(\gk^i) = \unm\cas_{\chi_i} + \unc I_{\pg_i} = \unc\sum_\alpha(\ad_{\pg_i}{e^i_\alpha})^2 + \unm I_{\pg_i}, \qquad i=1,2,
\end{equation}
where 
$$
\cas_{\chi_i}:=\cas_{\pg_i,-\kil_{\ggo_i}|_{\pi_i(\kg)}}:\pg_i\longrightarrow\pg_i
$$ 
is the Casimir operator of the isotropy representation $\chi_i:\pi_i(K)\rightarrow\End(\pg_i)$ of $G_i/\pi_i(K)$ with respect to the bi-invariant inner product $-\kil_{\ggo_i}|_{\pi_i(\kg)}$.  Note that $\cas_{\chi_i}\geq 0$, where equality holds if and only if $\pg_i=0$ (i.e., $M_i$ is a point).

\begin{proposition}\label{ricggals2}\cite[Proposition 3.2]{BRF}
The Ricci operator of a metric $g=(x_1,x_2,x_3)$ on an aligned homogeneous space $M=G/K$ with positive constants $c_1,\lambda_1,\dots,\lambda_t$ is given by 
\begin{enumerate}[{\rm (i)}]
\item $\Ricci(g)|_{\pg_1} = 
\tfrac{1}{2x_1}  \left(1 - \tfrac{(c_1-1)x_3}{c_1x_1}\right) \cas_{\chi_1}
+ \tfrac{1}{4x_1}I_{\pg_1}$.  
\item[ ]
\item $\Ricci(g)|_{\pg_2} = 
\tfrac{1}{2x_2} \left(1 - \tfrac{x_3}{c_1x_2} \right) \cas_{\chi_2} 
+ \tfrac{1}{4x_2}I_{\pg_2} $.
\item[ ]
\item The decomposition $\pg=\pg_1\oplus\pg_2\oplus\pg_3^0\oplus\dots\oplus\pg_3^t$ is $\ricci(g)$-orthogonal.
\item[ ]
\item $\Ricci(g)|_{\pg_3^l}= r_{3,l}I_{\pg_3^l}$, $\quad l=0,1,\dots,t$, where 
\begin{align*}
r_{3,l} 
:=& \tfrac{(c_1-1)\lambda_l}{4x_3}\left(\tfrac{c_1^2}{(c_1-1)^2}-\tfrac{x_3^2}{x_1^2} 
-\tfrac{x_3^2}{(c_1-1)^2x_2^2}\right) 
+\tfrac{c_1-1}{4x_3}\left(\tfrac{x_3^2}{c_1x_1^2} 
+\tfrac{x_3^2}{c_1(c_1-1)x_2^2}\right).
\end{align*} 
\end{enumerate}
\end{proposition}

\begin{proof}
We use the notation in \cite[Sections 2 and 3]{BRF} and the formula for the Ricci curvature of aligned homogeneous spaces given in \cite[Proposition 3.2]{BRF}, where $A_3=-\tfrac{c_2z_1}{c_1z_2}$ and $B_3=\tfrac{z_1}{c_1}+A_3^2\tfrac{z_2}{c_2}$ .  Since we are considering here $g_b=\gk$, i.e., $z_1=z_2=1$, we have that 
$A_3=-\tfrac{1}{c_1-1}=-B_3$ (recall that $c_2=\tfrac{c_1}{c_1-1}$), thus the proposition is a direct application of the formulas given in \cite[Proposition 3.2]{BRF}, except for the formula for $r_{3,l}$, which is obtained as follows: 
\begin{align*}
r_{3,l} =& \tfrac{\lambda_l}{4B_3x_3}\left(\tfrac{2x_1^2-x_3^2}{x_1^2} 
+\tfrac{(2x_2^2-x_3^2)A_3^2}{x_2^2}  
-\tfrac{1+A_3}{B_3}\left(\tfrac{1}{c_1}+\tfrac{1}{c_2}A_3^3\right)\right) \\ 
&+\tfrac{1}{4B_3x_3}\left(2\left(\tfrac{1}{c_1}+\tfrac{1}{c_2}A_3^2\right) 
-\tfrac{2x_1^2-x_3^2}{c_1x_1^2} 
- \tfrac{(2x_2^2-x_3^2)A_3^2}{c_2x_2^2}\right)\\ 
=& \tfrac{(c_1-1)\lambda_l}{4x_3}\left(\tfrac{2x_1^2-x_3^2}{x_1^2} 
+\tfrac{2x_2^2-x_3^2}{(c_1-1)^2x_2^2}  
-(c_1-2)\left(\tfrac{1}{c_1}-\tfrac{1}{c_1(c_1-1)^2}\right)\right) \\ 
&+\tfrac{c_1-1}{4x_3}\left(2\left(\tfrac{1}{c_1}+\tfrac{1}{c_1(c_1-1)}\right) 
-\tfrac{2x_1^2-x_3^2}{c_1x_1^2} 
- \tfrac{2x_2^2-x_3^2}{c_1(c_1-1)x_2^2}\right)\\ 
%
%
=& \tfrac{(c_1-1)\lambda_l}{4x_3}\left(\tfrac{2x_1^2-x_3^2}{x_1^2} 
+\tfrac{2x_2^2-x_3^2}{(c_1-1)^2x_2^2}  
-\tfrac{(c_1-2)^2}{(c_1-1)^2}\right) 
+\tfrac{c_1-1}{4x_3}\left(\tfrac{2}{c_1-1} 
-\tfrac{2x_1^2-x_3^2}{c_1x_1^2} 
- \tfrac{2x_2^2-x_3^2}{c_1(c_1-1)x_2^2}\right) \\
=& \tfrac{(c_1-1)\lambda_l}{4x_3}\left(\tfrac{c_1^2}{(c_1-1)^2}-\tfrac{x_3^2}{x_1^2} 
-\tfrac{x_3^2}{(c_1-1)^2x_2^2}\right) 
+\tfrac{c_1-1}{4x_3}\left(\tfrac{x_3^2}{c_1x_1^2} 
+\tfrac{x_3^2}{c_1(c_1-1)x_2^2}\right),
\end{align*}   
concluding the proof.  
\end{proof}

\section{Structural constants}\label{ijk-sec}
We provide in this section an alternative proof of the formula for the Ricci curvature of an aligned homogeneous space $M=G_1\times G_2/\Delta K$ given in Proposition \ref{ricggals2}, in the case when the existence of an  Einstein metric of the form $g=(x_1,x_2,x_3)$ is possible.  We therefore make the following assumption in this section:

\begin{assum}
$\cas_{\chi_1}=\kappa_1I_{\pg_1}$ and $\cas_{\chi_2}=\kappa_2I_{\pg_2}$ for some $\kappa_1,\kappa_2>0$ and either $K$ is semisimple and $\kil_{\pi_1(\kg)}=a_1\kil_{\ggo_1}|_{\pi_2(\kg)}$ and $\kil_\kg=a_2\kil_{\ggo_2}|_\kg$ (i.e., $\lambda_1=\dots=\lambda_t=:\lambda$ and the construction given in Example \ref{kill-exa} applies) or $K$ is abelian (i.e., $\lambda=0$).  
\end{assum}

Given any homogeneous space $G/K$ and the $Q$-orthogonal reductive decomposition $\ggo=\kg\oplus\pg$ with respect to a bi-invariant inner product $Q$ on $\ggo$, the so called {\it structural constants} of a $Q$-orthogonal decomposition
$
\pg=\pg_1\oplus\dots\oplus\pg_r
$ 
in $\Ad(K)$-invariant  subspaces (not necessarily $\Ad(K)$-irreducible) are defined by
\begin{equation}\label{ijk-gen}
[ijk]:=\sum_{\alpha,\beta,\gamma} Q([e_{\alpha}^i,e_{\beta}^j], e_{\gamma}^k)^2,
\end{equation}
where $\{ e_\alpha^i\}$, $\{ e_{\beta}^j\}$ and $\{ e_{\gamma}^k\}$ are $Q$-orthonormal basis of $\pg_i$, $\pg_j$ and $\pg_k$, respectively.  

\begin{lemma}\label{ijk}
The nonzero structural constants of the $\gk$-orthogonal reductive complement $\pg=\pg_1\oplus\pg_2\oplus\pg_3$ are given by 
$$
\begin{array}{c}
[111]=(1-2\kappa_1)n_1, \qquad [222]=(1-2\kappa_2)n_2, \qquad [333]=\tfrac{(c_1-2)^2\lambda d}{c_1-1}, \\

[113]=\tfrac{(c_1-1)\kappa_1n_1}{c_1} ,
\qquad 
[223]=\tfrac{\kappa_2n_2}{c_1}. 
\end{array}
$$
\end{lemma}

\begin{proof}
The union of the $\gk$-orthonormal basis $\{e^3_\alpha=\sqrt{c_1-1}(Z_1^\alpha,-\tfrac{1}{c_1-1}Z_2^\alpha)\}$ of $\pg_3$, where $\{ Z^\alpha\}$ is a $-\kil_\ggo$-orthonormal basis of $\kg$, and $\gk$-orthonormal bases $\{e_\alpha^i\}_{\alpha=1}^{\dim{\pg_i}}$ of $\pg_i$, $i=1,2$, form the $\gk$-orthonormal basis of $\pg$ which will be used in the computations.  
  
According to \eqref{MgBi}, for $i=1,2$,
$$
[iii] = \sum_{\alpha,\beta,\gamma} \gk([e^i_\alpha,e^i_\beta],e^i_\gamma)^2 
= -\sum_\alpha\tr{(\ad_{\pg_i}{e^i_\alpha})^2} = -2\tr{\cas_{\chi_i}}+\tr{I_{\pg_i}} = (1-2\kappa_i)n_i,
$$
and on the other hand, using that $-\kil_{\ggo_i}(Z_i^\alpha, Z_i^\beta)=\tfrac{1}{c_i}\delta_{\alpha\beta}$ by \eqref{al2}, we obtain
\begin{align*}
[113] =& \sum_{\alpha,\beta,\gamma} \gk([e^3_\alpha,e^1_\beta],e^1_\gamma)^2 
= (c_1-1)\sum_{\alpha,\beta,\gamma} \gk([Z_1^\alpha,e^1_\beta],e^1_\gamma)^2 
= (c_1-1)\sum_{\alpha} -\tr{(\ad{Z_1^\alpha})^2} \\
=&  (c_1-1)\tfrac{\tr{\cas_{\chi_1}}}{c_1} =(c_1-1) \tfrac{\kappa_1n_1}{c_1},\\
[223] =& \sum_{\alpha,\beta,\gamma} \gk([e^3_\alpha,e^2_\beta],e^2_\gamma)^2 
= \tfrac{1}{c_1-1}\sum_{\alpha,\beta,\gamma} \gk([Z_2^\alpha,e^2_\beta],e^2_\gamma)^2 
= \tfrac{1}{(c_1-1)}\sum_{\alpha} -\tr{(\ad{Z_2^\alpha})^2} \\
=&  \tfrac{1}{c_1-1}\tfrac{\tr{\cas_{\chi_2}}}{c_2} = \tfrac{\kappa_2n_2}{c_1}.  
\end{align*}

Finally, we have that 
\begin{align*}
[333] =& \sum_{\alpha,\beta,\gamma} \gk([e^3_\alpha,e^3_\beta],e^3_\gamma)^2  
= \sum_{\alpha} -\tr{(\ad{e^3_\alpha}|_{\pg_3})^2} \\
=& (c_1-2)\left(\tfrac{1}{c_1}-\tfrac{1}{c_1(c_1-1)^2}\right)\lambda \sum_{\alpha}\gk\left(\sqrt{c_1-1}Z^\alpha, \sqrt{c_1-1}Z^\alpha\right)\\ 
=& (c_1-2)\left(\tfrac{1}{c_1}-\tfrac{1}{c_1(c_1-1)^2}\right)\lambda (c_1-1)d 
=\tfrac{(c_1-2)^2\lambda d}{c_1-1},
\end{align*}
concluding the proof.  
\end{proof} 

\begin{corollary}\label{rics2str}
The Ricci curvature of the metric $g=(x_1,x_2,x_3)_{\gk}$ satisfies that $\ricci(g)(\pg_i,\pg_j)=0$ for all $i\ne j$ and $\Ricci(g)|_{\pg_i} = r_iI_{\pg_i}$, where 
$$
\begin{array}{c}
r_1
= \tfrac{1+2\kappa_1}{4}\tfrac{1}{x_1} - \tfrac{(c_1-1)\kappa_1}{2c_1}\tfrac{x_3}{x_1^2}, \qquad 
r_2 
= \tfrac{1+2\kappa_2}{4}\tfrac{1}{x_2} - \tfrac{\kappa_2}{2c_1}\tfrac{x_3}{x_2^2}, \\ \\ 
r_3 = \left(\tfrac{1}{2} - \tfrac{(c_1-1)(1-c_1\lambda)}{2c_1} - \tfrac{(c_1-1-c_1\lambda)}{2c_1(c_1-1)} - \tfrac{(c_1-2)^2\lambda }{4(c_1-1)}\right)\tfrac{1}{x_3} 
+ \tfrac{(c_1-1)(1-c_1\lambda)}{4c_1}\tfrac{x_3}{x_1^2} + \tfrac{c_1-1-c_1\lambda}{4c_1(c_1-1)}\tfrac{x_3}{x_2^2}.
\end{array}
$$
\end{corollary}

\begin{remark}
These formulas coincide with those provided in Proposition \ref{ricggals2}. 
\end{remark}

\begin{proof}
We use the well-known formula for the Ricci eigenvalues in terms of structural constants (see e.g.\ \cite[(18)]{stab-dos}) to obtain that 
\begin{align*}
r_1 =& \tfrac{1}{2x_1} - \tfrac{1}{4n_1}[111]\tfrac{1}{x_1} - \tfrac{1}{2n_1}[131]\tfrac{x_3}{x_1^2} 
= \left(\tfrac{1}{2} - \tfrac{1-2\kappa_1}{4}\right)\tfrac{1}{x_1} - \tfrac{(c_1-1)\kappa_1}{2c_1}\tfrac{x_3}{x_1^2} \\ 
=& \tfrac{1+2\kappa_1}{4}\tfrac{1}{x_1} - \tfrac{(c_1-1)\kappa_1}{2c_1}\tfrac{x_3}{x_1^2}, \\ 
r_2 =& \tfrac{1}{2x_2} - \tfrac{1}{4n_2}[222]\tfrac{1}{x_2} - \tfrac{1}{2n_2}[232]\tfrac{x_3}{x_2^2} 
= \left(\tfrac{1}{2} - \tfrac{1-2\kappa_2}{4}\right)\tfrac{1}{x_2} - \tfrac{\kappa_2}{2c_1}\tfrac{x_3}{x_2^2} \\
=& \tfrac{1+2\kappa_2}{4}\tfrac{1}{x_2} - \tfrac{\kappa_2}{2c_1}\tfrac{x_3}{x_2^2}, \\ 
r_3 =& \tfrac{1}{2x_3} - \tfrac{1}{4d}[113]\left(\tfrac{2}{x_3}-\tfrac{x_3}{x_1^2}\right) 
- \tfrac{1}{4d}[223]\left(\tfrac{2}{x_3}-\tfrac{x_3}{x_2^2}\right)- \tfrac{1}{4d}[333]\tfrac{1}{x_3} \\ 
=& \left(\tfrac{1}{2} - \tfrac{1}{2d}[113] - \tfrac{1}{2d}[223] - \tfrac{1}{4d}[333]\right)\tfrac{1}{x_3} 
+\tfrac{1}{4d}[113]\tfrac{x_3}{x_1^2} + \tfrac{1}{4d}[223]\tfrac{x_3}{x_2^2}\\ 
=& \left(\tfrac{1}{2} - \tfrac{(c_1-1)(1-c_1\lambda)}{2c_1} - \tfrac{(c_1-1-c_1\lambda)}{2c_1(c_1-1)} - \tfrac{(c_1-2)^2\lambda }{4(c_1-1)}\right)\tfrac{1}{x_3} 
+ \tfrac{(c_1-1)(1-c_1\lambda)}{4c_1}\tfrac{x_3}{x_1^2} + \tfrac{c_1-1-c_1\lambda}{4c_1(c_1-1)}\tfrac{x_3}{x_2^2}\\ 
=& \tfrac{2c_1(c_1-1)-2(c_1-1)^2(1-c_1\lambda)-2(c_1-1-c_1\lambda)-c_1(c_1-2)^2\lambda}{4c_1(c_1-1)} \tfrac{1}{x_3} 
+ \tfrac{(c_1-1)(1-c_1\lambda)}{4c_1}\tfrac{x_3}{x_1^2} + \tfrac{c_1-1-c_1\lambda}{4c_1(c_1-1)}\tfrac{x_3}{x_2^2}, 
\end{align*}
concluding the proof.  
\end{proof}

\section{Einstein metrics}\label{E-sec} 

In this section, we study the existence of Einstein metrics on aligned homogeneous spaces with two factors.  The case when $G_1=G_2$ and $K$ is diagonally embedded, i.e., $M=H\times H/\Delta K$ for some homogeneous space $H/K$, has already been considered in \cite{HHK}.   

\begin{theorem}\label{s2E}
On an aligned homogeneous space $M=G_1\times G_2/K$ with positive constants $c_1,\lambda_1,\dots,\lambda_t$, the metric 
$
g=(x_1,x_2,1)
$
is Einstein if and only if $\cas_{\chi_1}=\kappa_1I_{\pg_1}$ and $\cas_{\chi_2}=\kappa_2I_{\pg_2}$ for some $\kappa_1,\kappa_2>0$ and
\begin{enumerate}[{\rm (i)}] 
\item either $K$ is abelian and $x_1,x_2>0$ solve the following system of equations:
\begin{align}
c_1(2\kappa_1+1)x_1x_2^2=x_1^2+(c_1-1)(2\kappa_1+1)x_2^2, \label{E1} \\ 
c_1(2\kappa_2+1)x_1^2x_2 = (2\kappa_2+1)x_1^2+(c_1-1) x_2^2.  \label{E2}
\end{align} 

\item or $K$ is semisimple, $\lambda_1=\dots=\lambda_t=:\lambda$ and $x_1,x_2>0$ solve the following system of equations:
\begin{align}
&-c_1(2\kappa_2+1)x_1^2x_2+c_1(2\kappa_1+1)x_1x_2^2+2\kappa_2x_1^2-2(c_1-1)\kappa_1x_2^2=0, \label{E3} \\ 
&-c_1^3\lambda x_1^2x_2^2 +c_1(c_1 - 1) (2 \kappa_2 + 1)x_1^2x_2 \label{E4}\\
&+(c_1\lambda-(c_1 - 1) (2 \kappa_2 + 1))x_1^2 - (1-c_1\lambda) (c_1-1)^2 x_2^2=0. \notag
\end{align} 
\end{enumerate}
\end{theorem}

\begin{remark}\label{cand}
In order to admit an Einstein metric of this form, an aligned homogeneous space must therefore satisfy that the standard metric on both pieces $G_1/\pi_1(K)$ and $G_2/\pi_2(K)$ is Einstein and if $K$ is semisimple as in part (ii), then $\kil_{\pi_1(\kg)}=c_1\lambda\kil_{\ggo_1}|_{\pi_1(\kg)}$ and $\kil_{\pi_2(\kg)}=c_2\lambda\kil_{\ggo_2}|_{\pi_2(\kg)}$.  This implies that the space can be constructed as in Example \ref{kill-exa}, i.e., $M=G_1\times G_2/\Delta K$, from any two homogeneous spaces $G_1/K$ and $G_2/K$ such that their respective standard metrics are Einstein and $\kil_\kg=a_1\kil_{\ggo_1}|_\kg$ and $\kil_\kg=a_2\kil_{\ggo_2}|_\kg$, which have been listed in \cite[Tables 3-11]{HHK}.  There are $17$ infinite families and $50$ sporadic examples as possibilities for each $G_i/K$.  We assume from now on that $a_1\leq a_2$ (recall that $0<a_1,a_2<1$), which gives
$$
1<c_1=\tfrac{a_1+a_2}{a_2}\leq 2\leq c_2=\tfrac{a_1+a_2}{a_1}, \qquad \lambda=\tfrac{a_1a_2}{a_1+a_2}<\unm, 
\qquad c_1-1=\tfrac{a_1}{a_2}.
$$    
Recall that $\kappa_i=\tfrac{d(1-a_i)}{n_i}$, where $d=\dim{K}$ and $n_i=\dim{G_i}-d$.  
\end{remark}

\begin{remark}\label{hhk}
For $M=H\times H/\Delta K$, i.e., $a_1=a_2$, $\kappa_1=\kappa_2$ and $c_1=2$, it was proved in \cite{HHK} that there is exactly one solution if $K$ is abelian and the existence for $K$ semisimple is equivalent to 
$$
(2\kappa_1+1)^2\geq 8a_1(1-a_1+\kappa_1),
$$ 
which holds for most candidates $H/K$ listed in \cite[Tables 3-11]{HHK}.     
\end{remark}

\begin{remark}\label{dom}
Conditions \eqref{E1} and \eqref{E3} are both equivalent to $r_1=r_2$ (see Corollary \ref{rics2str}).  On the other hand, condition \eqref{E2} is precisely condition \eqref{E4} for $\lambda=0$ and they are equivalent to $r_2=r_3$.    
\end{remark}

\begin{proof}
Assume that $g$ is Einstein.  It follows from Proposition \ref{ricggals2}, (i) and (ii) that $\cas_{\chi_1}=\kappa_1I_{\pg_1}$ and $\cas_{\chi_2}=\kappa_2I_{\pg_2}$ for some $\kappa_1,\kappa_2>0$.  Moreover, we obtain the following formulas for the Ricci eigenvalues $r_1,r_2,r_{3,0},\dots,r_{3,t}$ of $g$ on $\pg_1,\pg_2,\pg_3^0,\dots,\pg_3^t$, respectively: 
\begin{align*}
r_1
&= \tfrac{1}{2x_1} \left(1 - \tfrac{c_1-1}{c_1x_1}\right)\kappa_1 + \tfrac{1}{4x_1}
=\tfrac{ c_1(2 \kappa_1 + 1) x_1  + 2\kappa_1 (1 - c_1)}{4 c_1 x_1^2},\\
r_2
=& \tfrac{1}{2x_2} \left(1 - \tfrac{1}{c_1x_2}\right)\kappa_2 + \tfrac{1}{4x_2} 
= \tfrac{c_1(2 \kappa_2 + 1) x_2  - 2 \kappa_2}{4 c_1 x_2^2},\\
r_{3,l} 
%
=& \tfrac{(c_1-1)\lambda_l}{4}\left(\tfrac{c_1^2}{(c_1-1)^2}-\tfrac{1}{x_1^2} 
-\tfrac{1}{(c_1-1)^2x_2^2}\right) 
+\tfrac{c_1-1}{4}\left(\tfrac{1}{c_1x_1^2} 
+\tfrac{1}{c_1(c_1-1)x_2^2}\right).
\end{align*} 
Thus the factor multiplying $\lambda_l$ in the formula for $r_{3,l}$ vanishes if and only if 
\begin{equation}\label{par}
x_1^2=\tfrac{ (c_1 - 1)^2 x_2^2}{c_1^2 x_2^2 - 1} \quad\mbox{and}\quad c_1x_2>1,
\end{equation} 
and so in that case, equation $r_2=r_3$ is equivalent to 
$$
\tfrac{c_1(2\kappa_2 + 1) x_2 - 2\kappa_2}{4c_1 x_2^2}   = \tfrac{c_1^2 x_2^2 + c_1 - 2}{4(c_1 - 1) c_1 x_2^2}.
$$
This implies that $x_2 = \frac{(2 \kappa_2+1)( c_1 -1) - 1}{c_1}$ and so $c_1x_2\leq 1$, which contradicts \eqref{par}.  We therefore obtain from $r_{3,0}=\dots=r_{3,t}$ that either $K$ is abelian or $K$ is semisimple and $\lambda_1=\dots=\lambda_t$.    

On the other hand, it is straightforward to see that $r_1=r_2$ is equivalent to equation \eqref{E3}, and in the case when $K$ is abelian, we have that
$$
r_2= \tfrac{c_1(2 \kappa_2 + 1) x_2  - 2 \kappa_2}{4 c_1 x_2^2}=\tfrac{(c_1 - 1)x_1^2 + (c_1- 1)^2 x_2^2}{4 c_1 (c_1 - 1) x_1^2x_2^2}=r_{3,0},
$$    
if and only if condition \eqref{E2} holds.  It is easy to see that condition \eqref{E3} is equivalent to \eqref{E1} by using \eqref{E2}.  

It only remains to prove part (ii), that is, equation $r_2=r_3$ is equivalent to condition \eqref{E4}, where $r_3:=r_{3,1}=\dots=r_{3,t}$, which follows from the following manipulations: if we multiply equation $r_2=r_3$, given by, 
$$
\tfrac{c_1(2 \kappa_2 + 1) x_2  - 2 \kappa_2}{4 c_1 x_2^2} 
= \tfrac{(c_1-1)\lambda}{4}\left(\tfrac{c_1^2}{(c_1-1)^2}-\tfrac{1}{x_1^2} 
-\tfrac{1}{(c_1-1)^2x_2^2}\right) 
+\tfrac{c_1-1}{4}\left(\tfrac{1}{c_1x_1^2} 
+\tfrac{1}{c_1(c_1-1)x_2^2}\right),
$$
by the factor $4c_1(c_1-1)x_1^2x_2^2$, we obtain that 
$$
(c_1 - 1) x_1^2(c_1(2 \kappa_2 + 1) x_2  - 2 \kappa_2)
= \lambda\left(c_1^3x_1^2x_2^2- c_1(c_1-1)^2x_2^2 - c_1x_1^2\right) 
+(c_1-1)^2x_2^2 + (c_1-1)x_1^2,
$$
from which \eqref{E4} easily follows, concluding the proof. 
\end{proof}

As known, $G$-invariant Einstein metrics on a compact homogeneous space $M=G/K$ are precisely the critical points of the scalar curvature functional,
$$
\scalar:\mca_1^G\longrightarrow \RR,  
$$ 
where $\mca^G_1$ is the space of all unit volume $G$-invariant metrics on $M$.  An Einstein metric $g\in\mca_1^G$ is called $G$-{\it unstable} if $\scalar''_g(T,T)>0$ for some $T\in{\tca\tca_g^G}$, where $\tca\tca_g^G$ is the space of $G$-invariant TT-tensors (see \cite{stab,stab-dos} for further information).  The stability type of the Einstein metrics that Theorem \ref{s2E} may provide can be obtained following the lines of \cite{stab-dos} (see also \cite[Section 6]{HHK}).  Using the structural constants computed in Lemma \ref{ijk}, we obtain that if 
$$
\mca^{G,diag}:=\{g=(x_1,x_2,x_3):x_i>0\},
$$
then the Hessian of $\scalar:\mca^{G,diag}\rightarrow\RR$ at an Einstein metric $g_0=(x_1,x_2,x_3)$ with Einstein constant $\rho$ is given by $\Hess(\scalar)_{g_0}=2\rho I-L$, where 
$$
L=\tfrac{1}{c_1}\left[\begin{matrix} 
\tfrac{(c_1-1)\kappa_1}{x_1^2} & 0 & -\tfrac{(c_1-1)\kappa_1\sqrt{n_1}}{\sqrt{d}x_1^2}\\ 
0 & \tfrac{\kappa_2}{x_2^2} & -\tfrac{\kappa_2\sqrt{n_2}}{\sqrt{d}x_2^2}\\ 
-\tfrac{(c_1-1)\kappa_1\sqrt{n_1}}{\sqrt{d}x_1^2} & -\tfrac{\kappa_2\sqrt{n_2}}{\sqrt{d}x_2^2} & \tfrac{\kappa_2n_2x_1^2+(c_1-1)\kappa_1n_1x_2^2}{dx_1^2x_2^2}
\end{matrix}\right].
$$

\begin{proposition}\label{stab}
Any Einstein metric on $M=G_1\times G_2/K$ provided by Theorem \ref{s2E} is $G$-unstable.  
\end{proposition}

\begin{proof}
It follows from the proof of Theorem \ref{s2E} that $\rho=\tfrac{c_1(2\kappa_2+1)x_2-2\kappa_2}{4c_1x_2^2}$.  Using that $c_1x_2>1$ (see \eqref{c1x2} and \eqref{E5} below) and $\kappa_2\leq \unm$, we obtain that 
$$
2\rho-L_{22} =\tfrac{c_1(2\kappa_2+1)x_2-2\kappa_2}{2c_1x_2^2} - \tfrac{\kappa_2}{c_1x_2^2} 
= \tfrac{c_1(2\kappa_2+1)x_2-4\kappa_2}{2c_1x_2^2} > \tfrac{-2\kappa_2+1}{2c_1x_2^2} \geq 0.  
$$
Thus $2\rho-L|_{T_{g_0}\mca^{G,diag}_1}$ has at least one positive eigenvalue and the instability of these Einstein metrics as critical points of $\scalar:\mca^G_1\rightarrow\RR$ follows.     
\end{proof}

\begin{figure}\label{sec}
\includegraphics[width=47mm]{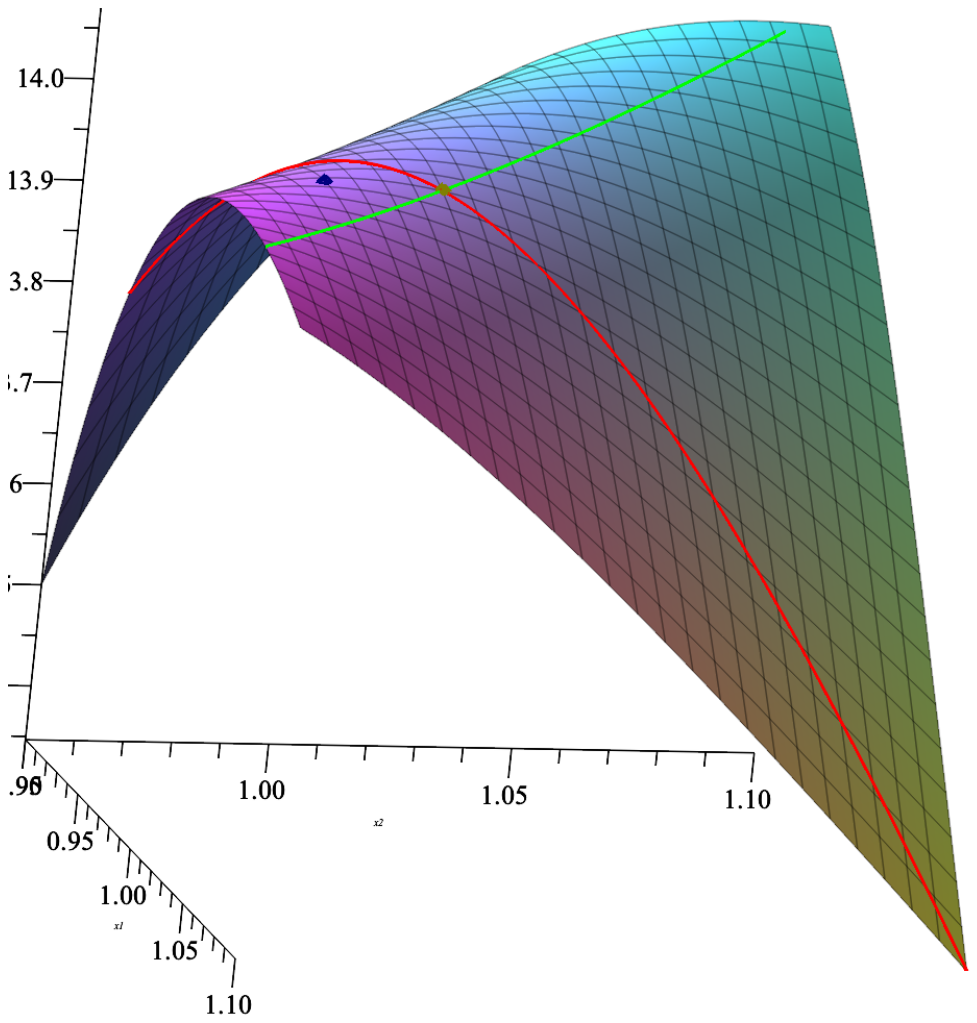}
\includegraphics[width=49mm]{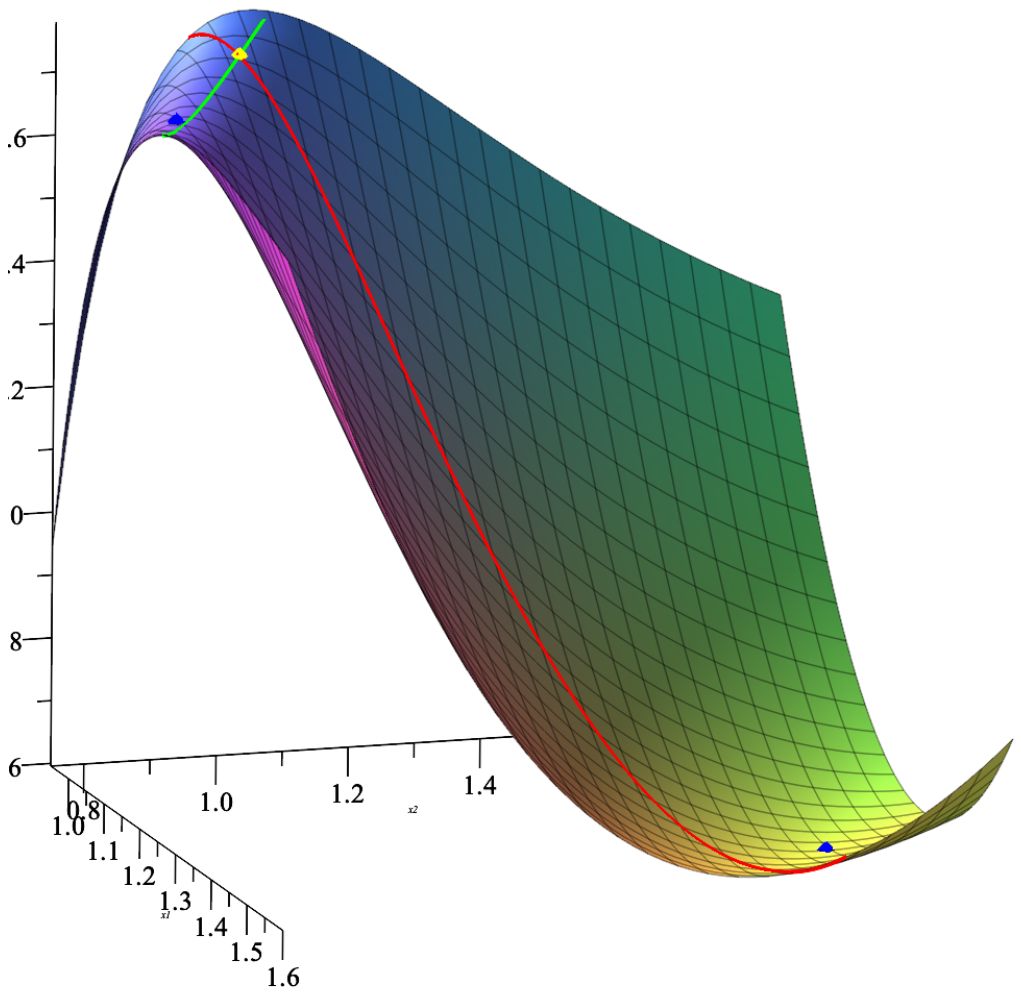}
\includegraphics[width=51mm]{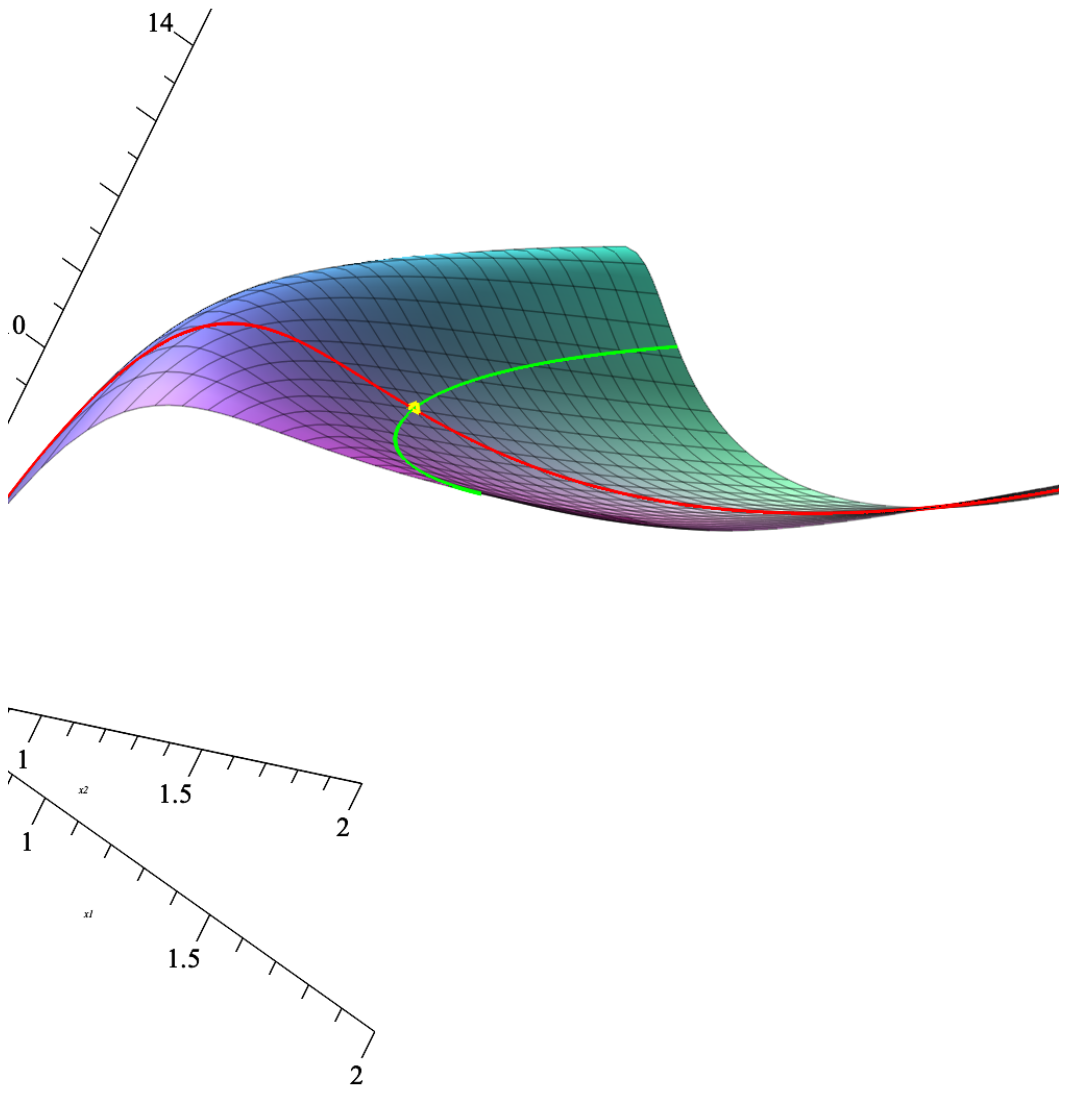}
\caption{Graph of $\scalar:\mca^G_1\rightarrow\RR$ in the variables $(x_1,x_2)$ for, from left to right, $M^{48}=\SU(5)\times\SO(8)/T^4$, $M^{21}=\Gg_2\times\Spe(2)/\SU(2)$ and $M^{29}=\SU(5)\times\SU(4)/\Spe(2)$, which admit one, two and none invariant {\color{blue} Einstein} metrics (i.e., critical points, in blue), respectively.  The {\color{yellow} standard} metric $\gk$ ($x_1=x_2=1$) is in yellow and belongs to both the green curve of {\color{green} normal} metrics and to the red curve defined by {\color{red} $x_1=x_2$}.}\label{fig}
\end{figure}

In Figure \ref{fig}, the graph of $\scalar:\mca^G_1\rightarrow\RR$ has been drawn for three examples.

\subsection{$K$ abelian}\label{Kab-sec} 
We need to analyze the existence problem for positive solutions to the algebraic equations given in Theorem \ref{s2E}, starting in this section with the case when $K$ is abelian.  

\begin{proposition}\label{EKab}
Any aligned homogeneous space $M=G_1\times G_2/K$ such that $K$ is abelian and $\cas_{\chi_1}=\kappa_1I_{\pg_1}$, $\cas_{\chi_2}=\kappa_2I_{\pg_2}$ for some $\kappa_1,\kappa_2>0$, admits exactly one Einstein metric of the form $g=(x_1,x_2,1)$, which is always a saddle point.  
\end{proposition}

\begin{remark}
Alternatively, the existence follows from the Simplicial Complex Theorem in \cite{Bhm}.  Indeed, it is easy to see that the intermediate subalgebras $\kg\oplus\pg_1$ and $\kg\oplus\pg_2$ belong to distinct non-toral components of the graph attached to $G_1\times G_2/K$.  
\end{remark}

\begin{proof}
It follows from \eqref{E2} that necessarily,
\begin{equation}\label{Eu2}
x_1=  \tfrac{\sqrt{c_1 - 1}x2}{ \sqrt{(2\kappa_2+1)(c_1 x_2 -1)}}, 
\end{equation}
from which \eqref{E1} becomes the following identity for $x_2$: 
$$
2\kappa_1(2\kappa_2+1)(c_1x_2-1) - \tfrac{c_1}{\sqrt{c_1-1}} (2\kappa_1+1) x_2\sqrt{ (2\kappa_2+1)(c_1x_2-1)} 
+ c_1(2\kappa_2+1)x_2  - 2\kappa_2=0. 
$$
If we set $u:=\sqrt{c_1x_2-1}$, then it is is easy to see that the above condition is equivalent to the cubic 
\begin{equation}\label{Eu}
q(u):=u^3 - \sqrt{(c_1-1)(2\kappa_2+1)}u^2 + u - \tfrac{ \sqrt{c_1-1}}{(2\kappa_1 + 1)\sqrt{2 \kappa_2 +1}} =0,
\end{equation}
which clearly admits al least one positive solution $u_0$ since $q(0)<0$.  Thus 
\begin{equation}\label{c1x2}
x_2 = \tfrac{u_0^2+1}{c_1}>\tfrac{1}{c_1},
\end{equation}
and so $x_1$ is well defined.  Using that $q'(u)=3u^2 - 2\sqrt{(c_1-1)(2\kappa_2+1)}u + 1$ never vanishes (note that its discriminant is $4((c_1-1)(2\kappa_2+1)-3)<0$), we conclude that $q$ has only one root. 

Concerning the type of critical point this metric is, we argue as in the proof of Proposition \ref{stab}.  Note first that 
\begin{align*} 
2 \rho - L_{3,3}=&\tfrac{c_1 (2 \kappa_2+1) x_2 -2\kappa_2}{2 c_1 x_2^2} - \tfrac{x_1^2 +(c_1-1)x_2^2}{c_1 x_1^2 x_2^2} 
= \tfrac{    ( c_1 (2 \kappa_2+1) x_2 -2\kappa_2) x_1^2 - 2 (x_1^2 +(c_1-1)x_2^2) }{2 c_1  x_1^2 x_2^2}\\ 
=& \tfrac{    ( (2 \kappa_2+1) c_1 x_2 -2\kappa_2-2) x_1^2 - 2 (c_1-1)x_2^2 }{2 c_1  x_1^2 x_2^2}
= \tfrac{    (  (2 \kappa_2+1)( c_1 x_2 -1) -1) x_1^2 - 2 x_2^2 (c_1-1) }{2 c_1  x_1^2 x_2^2}. 
\end{align*} 
Now using \eqref{Eu2} we obtain that 
\begin{align*}
2 \rho - L_{3,3}=&\tfrac{( (2\kappa_2 + 1) (c_1 x_2 - 1) - 1) (c_1 - 1) x_2^2     - 2 x_2^2 (c_1 - 1)(c_1 x_2 - 1) (2\kappa_2 + 1)}{2 c_1 x_1^2 x_2^2 (c_1 x_2 - 1) (2\kappa_2 + 1)} \\ 
=&\tfrac{ (c_1 - 1) x_2^2(  (2\kappa_2 + 1) (c_1 x_2 - 1) - 1  - 2(c_1 x_2 - 1) (2\kappa_2 + 1))}{2 c_1 x_1^2 x_2^2 (c_1 x_2 - 1) (2\kappa_2 + 1)} \\
=&- \tfrac{ (c_1 - 1) x_2^2(  (2\kappa_2 + 1) (c_1 x_2 - 1) + 1)}{2 c_1 x_1^2 x_2^2 (c_1 x_2 - 1) (2\kappa_2 + 1)} < 0.
\end{align*}
This implies that $2\rho-L|_{T_{g_0}\mca^{G,diag}_1}$ has at least one negative eigenvalue, which combined with Proposition \ref{stab} gives that the Einstein metric is a saddle point of $\scalar:\mca^G_1\rightarrow\RR$, as was to be shown.    
\end{proof}

The class involved in the above corollary is not that large, it can be obtained from \cite[Table 8]{HHK} and consists of 
\begin{enumerate}[\small{$\bullet$}]
\item $\SU(m+1)\times\SO(2m)/T^m$, \quad $m\geq 4$, 

\item $\SU(2)\times\SU(2)/T^1$, \quad $\SU(6)\times \Eg_6/T^6$, \quad $\SU(7)\times \Eg_7/T^7$, \quad $\SU(8)\times \Eg_8/T^8$, 

\item $\SO(12)\times \Eg_6/T^6$, \quad $\SO(14)\times \Eg_7/T^7$, \quad $\SO(16)\times \Eg_8/T^8$.  
\end{enumerate}
Each one is actually an infinite family of homogeneous spaces since the torus can be embedded in $G_1\times G_2$ with any slope $(p,q)$, $p,q\in\NN$, which gives $c_1=\tfrac{p^2+q^2}{p^2}$ in much the same way as in Example \ref{dim5}.  

\begin{example}\label{48} 
Consider the space $M^{48}=\SU(5)\times\SO(8)/T^4$ with $c_1=2$ (i.e., $p=q$), which has $n_1 = 11$, $n_2 = 7$, $d = 4$, $\kappa_1=\frac{1}{5}$, $\kappa_2=\frac{1}{6}$.  The cubic in \eqref{Eu} is given by 
$$
q(u)= u^3 - \tfrac{2}{\sqrt{3}} u^2 + u  - \tfrac{5}{14\sqrt{3}}, 
$$
and has discriminant $\Delta(q)=-\tfrac{2323}{588}< 0$.  Thus there is exactly one real root, which is given by 
$$
u_0 = \tfrac{c}{126} - \tfrac{70}{3c} +\tfrac{2}{3 \sqrt{3}} \approx 0.8405, \qquad c=(200802 \sqrt{3} + 7938\sqrt{2323} )^{\tfrac{1}{3}},
$$ 
and so $g\approx (0.8791,0.8532,1)$ (see Figure \ref{fig}). 
\end{example}

\subsection{$K$ semisimple}\label{Kss-sec}
In this section, we consider the case of an aligned homogeneous space $M=G_1\times G_2/K$ with $K$ semisimple such that $\cas_{\chi_1}=\kappa_1I_{\pg_1}$ and $\cas_{\chi_2}=\kappa_2I_{\pg_2}$ for some $\kappa_1,\kappa_2>0$.  According to Theorem \ref{s2E} and Remark \ref{cand}, if the Killing constants are $a_1,a_2$ (i.e., $c_1=\tfrac{a_1+a_2}{a_2}$, $\lambda=\tfrac{a_1a_2}{a_1+a_2}$, $\kappa_i=\tfrac{d(1-a_i)}{n_i}$), then the  Einstein equations for the metric $g=(x_1,x_2,1)_{\gk}$ can be written as
\begin{align}
Ax_1^2x_2+Bx_1x_2^2+Cx_1^2+Dx_2^2=&0, \label{E6} \\ 
Ex_1^2x_2^2 +Fx_1^2x_2 +Gx_1^2 + Hx_2^2=&0, \label{E7}
\end{align} 
where
$$
A:=-c_1(2\kappa_2+1)<0, \quad B:=c_1(2\kappa_1+1)>0, \quad C:=2\kappa_2>0, \quad D:=-2(c_1-1)\kappa_1<0, 
$$
$$
E:=-c_1^3\lambda<0, \qquad F:=c_1(c_1 - 1) (2 \kappa_2 + 1)>0, 
$$
$$
G:=c_1\lambda-(c_1 - 1) (2 \kappa_2 + 1)<0, 
\qquad H:= - (1-c_1\lambda) (c_1-1)^2<0.
$$
Note that $G<0$ by condition \eqref{E5} below.   

\begin{proposition}\label{KssE}
A metric $g=(x_1,x_2,1)_{\gk}$ on $M=G_1\times G_2/K$ with $K$ semisimple such that $\cas_{\chi_i}=\kappa_iI_{\pg_i}$, $i=1,2$, is Einstein if and only if $x_2$ is a root of the quartic polynomial 
\begin{equation}\label{pE}
p(x)= ax^4 +bx^3 + cx^2 + d x +e, 
\end{equation}
and $x_1^2=\tfrac{-Hx_2^2}{Ex_2^2 +Fx_2 +G}$, where 
$$
a:=D^2E^2+B^2EH>0, \quad b:=B^2FH - 2DE(AH-DF)<0, 
$$
$$
c:=(AH-DF)^2+2DE(DG-CH)+B^2GH>0, 
$$
$$
d:=-2(AH-DF)(DG-CH)<0,  
\qquad
e:= (DG-CH)^2>0.
$$
In that case, 
\begin{equation}\label{E5}
\tfrac{1}{c_1}<x_2<\tfrac{(c_1-1)(2\kappa_2+1)-c_1\lambda}{c_1^2\lambda} 
=\tfrac{c_1G}{E}.   
\end{equation}
\end{proposition}

\begin{proof}
We consider the quadratic polynomial 
$$
q(x):=Ex^2 +Fx +G=(c_1x-1)((c_1-1)(2\kappa_2+1)-c_1\lambda(c_1x+1)).  
$$
It follows from Remark \ref{cand} that its two roots satisfy
\begin{equation}\label{pE3}
\tfrac{1}{c_1}<\tfrac{(c_1-1)(2\kappa_2+1)-c_1\lambda}{c_1^2\lambda} \quad\mbox{if and only if}\quad 
a_2<\tfrac{2d+n_2}{2d+2n_2}, 
\end{equation}
which always hold by \cite[Theorem 1]{DtrZll}.  Thus condition \eqref{E5} follows from the fact that $q(x_2)>0$ by \eqref{E7}.  

If $g$ is Einstein, then by \eqref{E7}, $q(x_2)>0$ and 
$
x_1^2=\tfrac{-Hx_2^2}{q(x_2)}. 
$
It now follows from \eqref{E6} that 
\begin{align*}
x_1=&\tfrac{1}{Bx_2^2} (-Ax_1^2x_2-Cx_1^2-Dx_2^2) 
=\tfrac{1}{Bx_2^2} (-x_1^2(Ax_2+C)-Dx_2^2) \\ 
=&\tfrac{1}{Bx_2^2} \left(\tfrac{Hx_2^2}{q(x_2)}(Ax_2+C)-Dx_2^2\right) 
=\tfrac{H(Ax_2+C)-Dq(x_2)}{Bq(x_2)},
\end{align*}
which implies that 
\begin{align*}
\tfrac{-Hx_2^2}{q(x_2)} 
= \left(\tfrac{H(Ax_2+C)-D(Ex_2^2 +Fx_2 +G)}{Bq(x_2)}\right)^2.
\end{align*}
This is equivalent to
\begin{align}
&-B^2Hx_2^2(Ex_2^2 +Fx_2 +G) = \left(H(Ax_2+C)-D(Ex_2^2 +Fx_2 +G)\right)^2 \label{pE2}\\ 
=& H^2(Ax_2+C)^2 + D^2(Ex_2^2 +Fx_2 +G)^2 - 2DH(Ax_2+C)(Ex_2^2 +Fx_2 +G) \notag\\ 
=& H^2\left(A^2x_2^2+C^2+2ACx_2\right) + D^2\left(E^2x_2^4 +F^2x_2^2 +G^2+2EFx_2^3+2EGx_2^2+2FGx_2\right)\notag \\ 
& - 2DH\left(AEx_2^3+(AF+CE)x_2^2+(AG+CF)x_2+CG\right),\notag
\end{align}
which is easily checked to be precisely $p(x_2)=0$.  

Conversely, we assume that $p(x_2)=0$ for some $x_2\in\RR$ (in particular, $x_2\ne 0$).  It follows from \eqref{pE2} that $q(x_2)\geq 0$, where equality holds if and only if $x_2=-\tfrac{C}{A}=\tfrac{2\kappa_2}{c_1(2\kappa_2+1)}<\tfrac{1}{c_1}$, a contradiction by \eqref{pE3}.  Thus $q(x_2)>0$ and if we set $x_1^2=\tfrac{-Hx_2^2}{q(x_2)}$, then \eqref{E6} and \eqref{E7} hold and hence $g$ is Einstein, concluding the proof.  
\end{proof}

According to Proposition \ref{KssE}, Einstein metrics of the form $g=(x_1,x_2,1)_{\gk}$ are in one-to-one correspondence with the real roots of the quartic polynomial $p$ given in \eqref{pE}, which can be analyzed by considering its discriminant 
\begin{align*}
\Delta ={}&256a^{3}e^{3}-192a^{2}bde^{2}-128a^{2}c^{2}e^{2}+144a^{2}cd^{2}e-27a^{2}d^{4}\\
&+144ab^{2}ce^{2}-6ab^{2}d^{2}e-80abc^{2}de+18abcd^{3}+16ac^{4}e\\
&-4ac^{3}d^{2}-27b^{4}e^{2}+18b^{3}cde-4b^{3}d^{3}-4b^{2}c^{3}e+b^{2}c^{2}d^{2},
\end{align*}
and other three invariants given by, 
$$
R:= 64a^{3}e-16a^{2}c^{2}+16ab^{2}c-16a^{2}bd-3b^{4},\quad S:=8ac-3b^2,  \quad T:= b^3+8a^2d-abc.
$$
The following results on the nature of the roots of $p$ are well known (see \cite{Lzr, Rs}):
\begin{enumerate}[(i)]
\item $\Delta<0$: two different real roots and two non-real complex roots.  

\item $\Delta>0$:
\begin{enumerate}[a)]
\item $R<0$ and $S<0$: four different real roots. 

\item $R\geq 0$ or $S\geq 0$: no real roots.  
\end{enumerate}

\item $\Delta=0$:
\begin{enumerate}[a)]
\item $S\leq 0$ or $T\ne 0$: at least one real root. 

\item $S>0$ and $T= 0$: no real roots.  
\end{enumerate}
\end{enumerate}

\begin{figure}\label{sec}
\includegraphics[width=47mm]{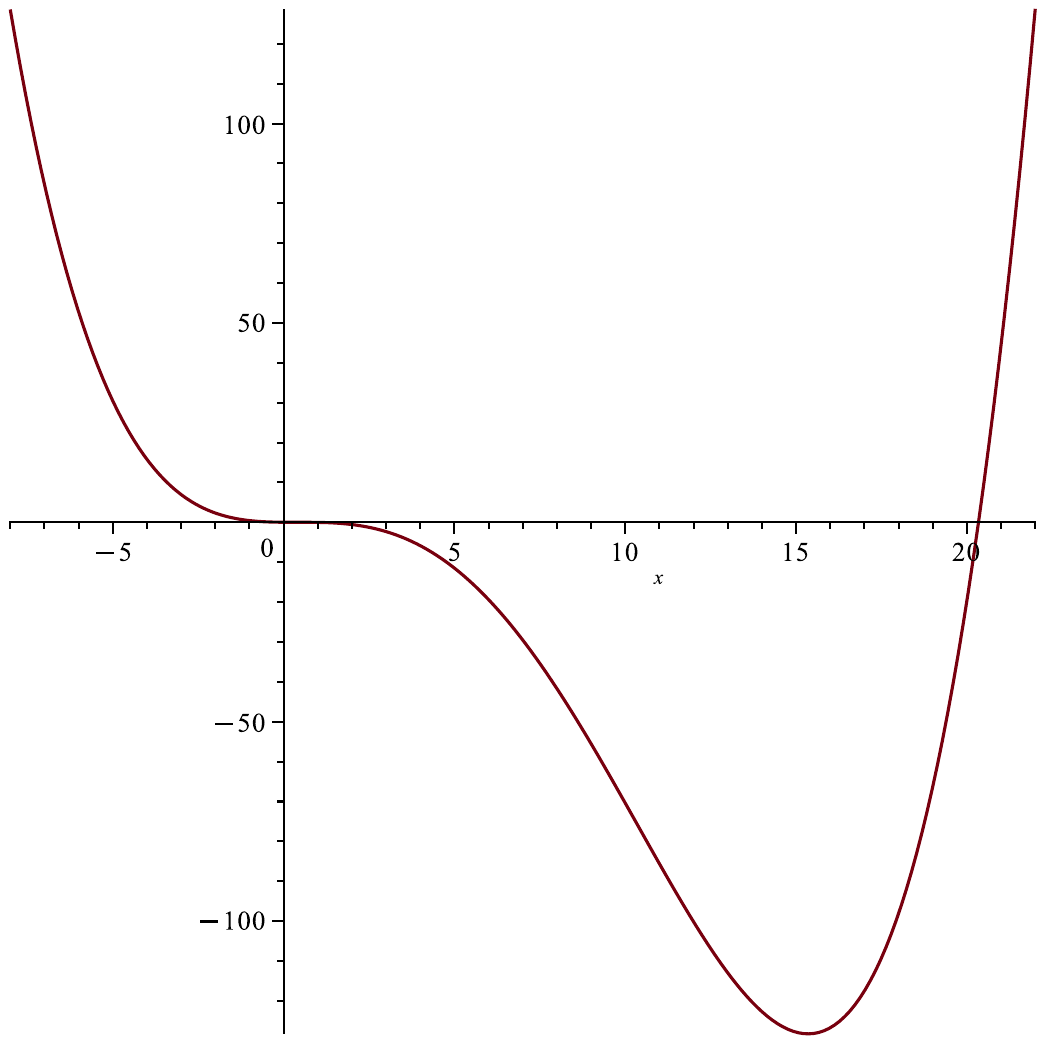}
\includegraphics[width=47mm]{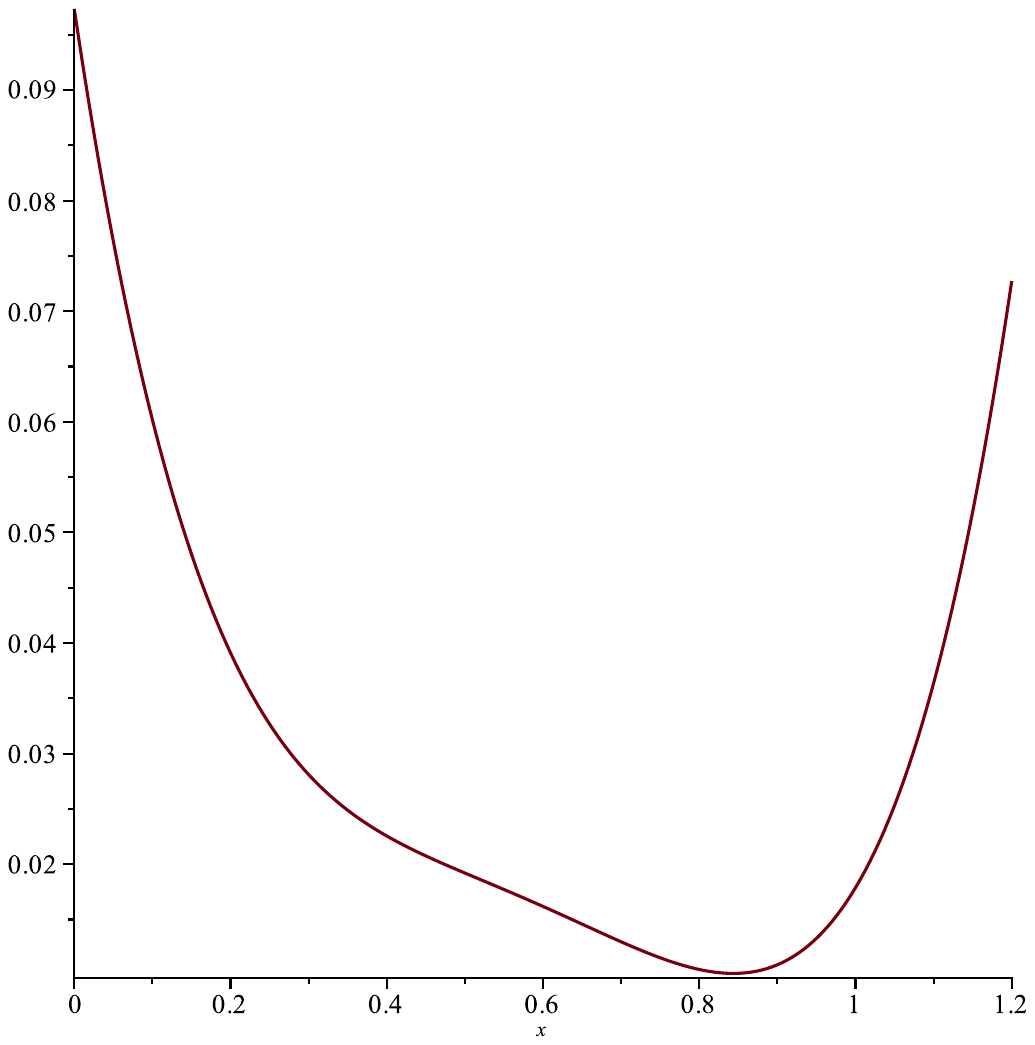}
\caption{Graph of the quartic polynomial $p$ whose roots are in bijection with invariant Einstein metrics on $M^{21}=\Gg_2\times\Spe(2)/\SU(2)$ (left) and $M^{29}=\SU(5)\times\SU(4)/\Spe(2)$ (right), which admit two and none, respectively.}\label{fig2}
\end{figure}

\begin{table} {\tiny
$$
\begin{array}{c|c|c}
K & d & G_i
\\[2mm] \hline \hline \rule{0pt}{14pt}
\SU(2) & 3 & \SU(3)_{5,\tfrac{1}{6}}, \quad \Spe(2)_{7,\tfrac{1}{15}}, \quad {\Gg_2}_{11,\tfrac{1}{56}}
\\[2mm] \hline  \rule{0pt}{14pt}
\SU(3) & 8 & {\Gg_2}_{6,\tfrac{3}{4}}, \quad \SO({\bf 8})_{20,\tfrac{1}{6}}, \quad \SU(6)_{27,\tfrac{1}{10}}, 
\\
&&{\Eg_6}_{70,\tfrac{1}{36}}, \quad {\Eg_7}_{125,\tfrac{1}{126}}
\\[2mm] \hline  \rule{0pt}{14pt}
\Gg_2 & 14 & {\SO(7)}_{7,\tfrac{4}{5}}, \quad {\Eg_6}_{64,\tfrac{1}{9}}, \quad \SO({\bf 14})_{77,\tfrac{1}{12}}
\\[2mm] \hline  \rule{0pt}{14pt}
\Spe(3) & 21 & \SO(14)_{70,\tfrac{13}{18}}, \quad \SU(6)_{14,\tfrac{2}{3}},  \quad \underline{\Spe(7)}_{84,\tfrac{1}{10}}, \quad \SO({\bf 21})_{189,\tfrac{1}{19}}
\\[2mm] \hline  \rule{0pt}{14pt}
\SU(6) & 35 & \SU(15)_{189,\tfrac{1}{10}}, \quad \underline{\Spe(10)}_{175,\tfrac{1}{11}},  \quad \SU(21)_{405,\tfrac{1}{28}}, \quad  \SO({\bf 35})_{560,\tfrac{1}{33}}
\\[2mm] \hline  \rule{0pt}{14pt}
\SO(9) & 36 & \SO(10)_{9,\tfrac{7}{8}}, \quad \underline{\Fg_4}_{16,\tfrac{7}{9}}, \quad \SU(9)_{44,\tfrac{7}{18}}, \quad \underline{\SO(16)}_{84,\tfrac{1}{4}}, 
\\
&& \SO({\bf 36})_{594,\tfrac{1}{34}}, \quad \SO(44)_{910,\tfrac{1}{66}}
\\[2mm] \hline  \rule{0pt}{14pt}
\Spe(4) & 36 & \SU(8)_{27,\tfrac{5}{8}}, \quad \SO(27)_{315,\tfrac{23}{30}}, \quad \underline{\Eg_6}_{42,\tfrac{5}{12}}, 
\\
&& \SO({\bf 36})_{594,\tfrac{1}{34}}, \quad \underline{\SO(42)}_{825,\tfrac{1}{56}}
\\[2mm] \hline  \rule{0pt}{14pt} 
\SO(10) & 45 & \SO(11)_{10,\tfrac{8}{9}}, \quad \SU(10)_{54,\tfrac{2}{5}}, \quad \underline{\SU(16)}_{210,\tfrac{1}{8}}, 
\\
&& \SO({\bf 45})_{945,\tfrac{1}{43}},\quad  \SO(54)_{1386,\tfrac{1}{78}}, 
\\[2mm] \hline  \rule{0pt}{14pt}
\Fg_4 & 52 & {\Eg_6}_{26,\tfrac{3}{4}}, \quad {\SO(26)}_{273,\tfrac{1}{8}}, \quad \SO({\bf 52})_{1274,\tfrac{1}{50}}
\\[2mm] \hline  \rule{0pt}{14pt}
\SU(8)  & 63 & \underline{\Eg_7}_{70,\tfrac{4}{9}}, \quad \SU(28)_{720,\tfrac{1}{21}}, \quad \SU(36)_{1232,\tfrac{1}{45}}, 
\\
&& \SO({\bf 63})_{1890,\tfrac{1}{61}}, \quad \underline{\SO(70)}_{2352,\tfrac{1}{85}}
\\[2mm] \hline  \rule{0pt}{14pt}
\SO(12) & 66 & \SO(13)_{12,\tfrac{10}{11}}, \quad \SU(12)_{77,\tfrac{5}{12}},  \quad \underline{\Spe(16)}_{462,\tfrac{5}{68}}, \\
&& \SO({\bf 66})_{2079,\tfrac{1}{64}}, \quad  \SO(77)_{2860,\tfrac{1}{105}}
\\[2mm] \hline  \rule{0pt}{14pt}
\Eg_6 & 78 & {\SU(27)}_{650,\tfrac{2}{27}}, \quad \SO({\bf 78})_{2925,\tfrac{1}{76}}
\\[2mm] \hline  \rule{0pt}{14pt}
\SU(9) & 80 & \underline{\Eg_8}_{168,\tfrac{3}{10}}, \quad \SU(36)_{1215,\tfrac{1}{28}}, \quad \SU(45)_{1944,\tfrac{1}{55}}, \quad  \SO({\bf 80})_{3080,\tfrac{1}{78}}
\\[2mm] \hline  \rule{0pt}{14pt}
\SO(16) & 120 &  \SO(17)_{16,\tfrac{14}{15}}, \quad  \underline{\Eg_8}_{128,\tfrac{7}{15}}, \quad \SU(16)_{135,\tfrac{7}{16}}, \quad \SO({\bf 120})_{7020,\tfrac{1}{118}}, 
\\
&& \underline{\SO(128)}_{8008,\tfrac{1}{144}}, \quad \SO(135)_{8925,\tfrac{1}{171}},
\\[2mm] \hline  \rule{0pt}{14pt}
\Eg_7 & 133 & {\Spe(28)}_{1463,\tfrac{3}{58}}, \quad \SO({\bf 133})_{8645,\tfrac{1}{131}}   
\\[2mm] \hline \rule{0pt}{14pt}
\Eg_8 & 248 & \SO({\bf 248})   
\\[2mm] \hline \rule{0pt}{14pt}
\SO(m), m\geq 5 & \tfrac{m(m-1)}{2} & \SO(m+1), \quad \SU(m), \quad \SO({\bf \tfrac{m(m-1)}{2}}), \quad \SO(\tfrac{(m-1)(m+2)}{2}),  
\\[2mm] \hline  \rule{0pt}{14pt}
\SU(m), m\geq 4 & m^2-1 & \SU(\tfrac{m(m-1)}{2}), \quad \SU(\tfrac{m(m+1)}{2}), \quad \SO({\bf m^2-1})  
\\[2mm] \hline  \rule{0pt}{14pt}
\Spe(m), m\geq 3 & m(2m+1) & \SU(2m), \quad \SO((m-1)(2m+1)), \quad \SO({\bf m(2m+1)})  
\\[2mm] \hline\hline
\end{array}
$$}
\caption{Isotropy irreducible homogeneous spaces $G_i/K$ with $K$ simple (see \cite[Tables 3,4,5,6,7,9]{HHK}).  For each $K$ appearing in any of the last three families, the extra $G_i$'s are underlined.  We denote by $\SO({\bf d})$ the group on which $K^d$ is embedded via the adjoint representation.  The notation ${G_i}_{n_i,a_i}$ means that $\dim{G_i/K}=n_i$ and $\kil_\kg=a_i\kil_{\ggo_i}$, e.g., in the third line, $\SO(8)_{20,\tfrac{1}{6}}$ means that $\dim{\SO(8)/\SU(3)}=20$ and $\kil_{\sug(3)}=\tfrac{1}{6}\kil_{\sog(8)}$.}  \label{ii-simple}
\end{table}

In order to give an idea of the length of computations involved in deciding whether $p$ has a real root or not, we next work out three examples with the aid of Maple.   Note that any invariant metric is necessarily of the form $g=(x_1,x_2,1)_{\gk}$ up to scaling in the three cases (see \S\ref{mf-sec} below).   

\begin{example}\label{21}
For the space $M^{21}=\Gg_2\times\Spe(2)/\SU(2)$, we have that  
$$
n_1 = 11, \quad n_2 = 7, \quad d = 3, \quad  a_1 = \tfrac{1}{56}, \quad a_2 = \tfrac{1}{15}, 
$$
and so $c_1=\tfrac{71}{56}$, $\lambda=\tfrac{1}{71}$, $\kappa_1=\tfrac{15}{56}$ and $\kappa_2=\tfrac{2}{5}$.  A straightforward computation gives that 
\begin{align*}
p(x)=&\tfrac{371645834625}{48358655787008} x^4-\tfrac{15992045085375}{96717311574016} x^3
+\tfrac{18067869653625}{96717311574016} x^2 
-\tfrac{1649818125}{26985857024} x+ \tfrac{455625}{30118144} \\ 
\approx & \; 0.0076 x^4 - 0.1653 x^3 + 0.1868 x^2 - 0.0611 x + 0.0151, 
\end{align*}
and 
$$
\Delta=-0.000001495938639,\quad S=-0.07053475834, \quad R= -0.001656504408.
$$
Thus $p$ has exactly two real roots, that is, $M^{21}=\Gg_2\times\Spe(2)/\SU(2)$ admits exactly two invariant Einstein metrics by Proposition \ref{KssE} (see Figures \ref{fig} and \ref{fig2}).   
\end{example}

\begin{example}\label{29}
Consider $M^{29}=\SU(5)\times\SU(4)/\Spe(2)$, for which 
$$
n_1 = 14,\quad n_2 = 5,\quad d = 10,\quad  a_1 = \tfrac{3}{10},\quad a_2 = \tfrac{3}{4},
$$
and hence $c_1=\frac{7}{5}$,  $\lambda=\frac{3}{14}$, $\kappa_1=\frac{1}{2}$ and $\kappa_2=\frac{1}{2}$.  It is straightforward to see that 
\begin{align*}
p(x)=&\tfrac{223293}{390625} x^4-\tfrac{524104}{390625} x^3
+\tfrac{455406}{390625} x^2  -\tfrac{37128}{78125} x+ \tfrac{1521}{15625}\\ 
= & \; 0.57163008 x^4 - 1.34170624 x^3 + 1.16583936 x^2 - 0.4752384 x + 0.097344, 
\end{align*}
and 
$$
\Delta=0.0001962504947, \quad R= 0.1971272177, \quad S=-0.06909613037.
$$
This implies that $p$ has no real roots, that is, $M^{29}=\SU(5)\times\SU(4)/\Spe(2)$ does not admit an invariant Einstein metric by Proposition \ref{KssE} (see Figures \ref{fig} and \ref{fig2}).  
\end{example}

\begin{example}\label{suso}
Consider the space $M^n=\SU(m)\times\SO(m+1)/\SO(m)$, $m\geq 6$, for which it is easy to see that
$$
n=m^2+m-1, \quad d=\tfrac{m(m-1)}{2}, \quad n_1=\tfrac{(m-1)(m+2)}{2}, \quad n_2=m, \quad \kappa_1=\kappa_2=\unm,
$$
and
$$
a_1=\tfrac{m-2}{2m}, \quad a_2=\tfrac{m-2}{m-1}, \quad c_1=\tfrac{3m-1}{2m}, \quad \lambda=\tfrac{m-2}{3m-1}.
$$
With the aid of Maple, it is straightforward to see that 
$$
\Delta = \tfrac{4782969\, (m + 2)^4  (m - 1)^{12} ( m-2/3)^2 (m+1)^3 (m-1/3)^{12}}{4294967296 m^{44}} q_1(m),  
$$
$$
R=\tfrac{(m-1)^6 (3m-1)^{10}}{16777216 m^{32}} q_2(m), \qquad S= \tfrac{(3m-1)^6 (m-1)^4}{ 4096 m^{16} }  q_3(m),
$$
where $q_1,q_2,q_3$ are polynomials of degree $11$, $16$ and $6$, respectively, and that  $q_i(m)>0$ for any $m\geq 6$.  We therefore obtain that $\Delta,R,S>0$ and so $p$ has no real roots, which implies that these spaces do not admit invariant Einstein metrics by Proposition \ref{KssE}.        
\end{example}

The aligned homogeneous spaces that can be constructed as in the above two examples from the other irreducible symmetric spaces $H/K$ with $K$ simple (listed in \cite[Table 3]{HHK}) are all given in Table \ref{sym}.  The existence problem can be solved in much the same way as the above examples, obtaining that only one of these seven spaces admits an invariant Einstein metric.  

In the following example, we show that the existence problem is very sensitive to the embedding of $K$ in $G_1$ and $G_2$.  

\begin{example}\label{sp1}
Using the isotropy irreducible space $\Spe(2)/\SU(2)$ (see \cite[Table 6]{HHK}), we construct two aligned spaces 
$$
M_1=\SU(3)\times\Spe(2)/\Delta_1\SU(2), \qquad M_2=\SU(3)\times\Spe(2)/\Delta_2\SU(2),
$$
where $\pi_1(\SU(2))$ is the usual block $\SU(2)\subset\SU(3)$ for $M_1$ and it is the symmetric pair $\SO(3)\subset\SU(3)$ for $M_2$.  Thus in both cases, $n=15$, $d=3$, $n_1=5$, $n_2=7$ and the pair $(a_1,a_2)$ is respectively given by  $(\tfrac{2}{3},\tfrac{1}{15})$ and  $(\tfrac{1}{6}, \tfrac{1}{15})$.   Since the usual embedding does not satisfy that the Casimir operator is a multiple of the identity, it follows from Theorem \ref{s2E} that there is no Einstein metric of the form $g=(x_1,x_2,x_3)$ on the space $M_1$.  On the contrary, for $M_2$, the Casimir condition holds since both spaces are isotropy irreducible and it is straightforward to see that $\Delta(p)<0$, so there exists two Einstein metrics of the form $g=(x_1,x_2,x_3)$ on the space $M_2$. 
\end{example}

\begin{table} 
{\tiny 
$$
\begin{array}{c|c|c|c|c}
M=G/K & m & a_1 & a_2 &   
\\[2mm] \hline \hline \rule{0pt}{14pt}
\SO(\tfrac{(m-1)(m+2)}{2})\times\SO(m+1)/\SO(m) & \geq 5 &\tfrac{2}{(m+3)(m+2)} & \tfrac{m-2}{m-1} &  \exists, m\leq 8
\\[2mm]  \hline \rule{0pt}{14pt}
\SO(\tfrac{(m-1)(m+2)}{2})\times\SU(m)/\SO(m) & \geq 5  & \tfrac{2}{(m+3)(m+2)} & \tfrac{m-2}{2m} & \exists
\\[2mm]  \hline \rule{0pt}{14pt}
\SU(m)\times\SO(m+1)/\SO(m) & \geq 6 &\tfrac{m-2}{2m} & \tfrac{m-2}{m-1}  & \nexists
\\[2mm]  \hline \rule{0pt}{14pt}
\SO(\tfrac{m(m-1)}{2})\times\SO(m+1)/\SO(m) & \geq 6 & \tfrac{2}{m(m-1)-4} & \tfrac{m-2}{m-1}  & \nexists
\\[2mm]  \hline \rule{0pt}{14pt}
\SO(\tfrac{(m-1)(m+2)}{2})\times\SO(\tfrac{m(m-1)}{2})/\SO(m) & \geq 5  & \tfrac{2}{(m+3)(m+2)} & \tfrac{2}{m(m-1)-4} &  \exists 
\\[2mm]  \hline \rule{0pt}{14pt}
\SO(\tfrac{m(m-1)}{2})\times\SU(m)/\SO(m) & \geq 5  & \tfrac{2}{m(m-1)-4} & \tfrac{m-2}{2m} & \exists
\\[2mm]  \hline \rule{0pt}{14pt}
\SU(\tfrac{m(m+1)}{2})\times\SO(m^2-1)/\SU(m) & \geq 5  & \tfrac{2}{(m+1)(m+2)} & \tfrac{1}{m^2-3} & \exists
\\[2mm]  \hline \rule{0pt}{14pt}
\SU(\tfrac{m(m-1)}{2})\times\SO(m^2-1)/\SU(m) & \geq 5  & \tfrac{2}{(m-1)(m-2)} & \tfrac{1}{m^2-3} & \exists
\\[2mm]  \hline \rule{0pt}{14pt}
\SU(\tfrac{m(m+1)}{2})\times\SU(\tfrac{m(m-1)}{2})/\SU(m) & \geq 5  & \tfrac{2}{(m+1)(m+2)} & \tfrac{2}{(m-1)(m-2)} & \exists
\\[2mm]  \hline \rule{0pt}{14pt}
\SO(m(2m+1))\times\SU(2m)/\Spe(m) & \geq 3  & \tfrac{1}{m(2m+1)-2} & \tfrac{m+1}{2m} & \exists
\\[2mm] \hline\rule{0pt}{14pt}
\SU(2m)\times\SO((m-1)(2m+1))/\Spe(m) & \geq 3  &\tfrac{m+1}{2m} & a_m &\exists,  m\geq 10
\\[2mm]  \hline \rule{0pt}{14pt}
\SO(m(2m+1))\times\SO((m-1)(2m+1))/\Spe(m) & \geq 3 & \tfrac{1}{m(2m+1)-2} & a_m & \exists
\\[2mm] \hline\hline
\end{array}
$$}
\caption{All infinite {\bf families} that can be constructed from two different isotropy irreducible spaces $G_i/K$ with $K$ simple (see Table \ref{ii-simple}).  Here $a_m:=1-\tfrac{2m^3-3m^2-3m+2}{m(m^2-1)(2m-3)}$.}\label{flies}
\end{table}

\begin{table} 
{\tiny 
$$
\begin{array}{c|c|c|c|c|c|c|c|c|c}
M^n=G/K & n & d & n_1 & n_2 & a_1 & a_2 & c_1 & \lambda  &  
\\[2mm] \hline \hline \rule{0pt}{14pt}
\SU(5)\times \SU(4)/\Spe(2) & 29 & 10 & 14 & 5 & \tfrac{3}{10} & \tfrac{3}{4} & \tfrac{7}{5} & \tfrac{3}{14} & \nexists 
\\[2mm]  \hline \rule{0pt}{14pt}
\SU(9)\times \Fg_4/\SO(9) & 96 & 36 & 44 & 16 & \tfrac{7}{18} & \tfrac{7}{9} & \frac{3}{2} & \frac{7}{27} & \nexists
\\[2mm]  \hline \rule{0pt}{14pt}
\Eg_6\times\SU(8)/\Spe(4) & 105 & 36 & 42 & 27 & \tfrac{5}{12} & \tfrac{5}{8} & \frac{5}{3} & \frac{1}{4} & \nexists 
\\[2mm]  \hline \rule{0pt}{14pt}
\Fg_4\times\SO(10)/\SO(9) & 61 & 36 & 16 & 9 & \tfrac{7}{9} & \tfrac{7}{8} & \frac{17}{9} & \frac{7}{17} & \nexists 
\\[2mm]  \hline \rule{0pt}{14pt}
\SU(16)\times \Eg_8/\SO(16) & 383 & 120 & 135 & 128 & \tfrac{7}{16} & \tfrac{7}{15} & \frac{31}{16} & \frac{7}{31} & \exists 
\\[2mm]  \hline \rule{0pt}{14pt}
\Eg_8\times\SO(17)/\SO(16) & 264 & 128 & 16  & 120 & \tfrac{7}{15} & \tfrac{14}{15} & \frac{3}{2} & \frac{14}{45} & \nexists 
\\[2mm]  \hline \rule{0pt}{14pt}
\SU(m)\times\SO(m+1)/\SO(m) & & \tfrac{m(m-1)}{2} & \tfrac{(m-1)(m+2)}{2} &m &\tfrac{m-2}{2m} & \tfrac{m-2}{m-1} & \frac{3m-1}{2m} & \frac{m-2}{3m-1} & \nexists
\\[2mm] \hline\hline
\end{array}
$$}
\caption{All examples that can be constructed from two different irreducible {\bf symmetric} spaces $G_i/K$ with $K$ simple (see \cite[Table 3]{HHK}).  Here $\kappa_1=\kappa_2=\unm$, $a_i=\tfrac{2d-n_i}{2d}$, $c_1=\tfrac{a_1+a_2}{a_2}$ and $\lambda=\tfrac{a_1a_2}{a_1+a_2}$.  In the last line, $m\geq 6$ and $n=m^2+m-1$.} \label{sym}
\end{table}

\section{The class $\cca$}\label{mf-sec}
The isotropy representation of an aligned homogeneous space $M=G/K$ is {\it multiplicity-free} (i.e., the sum of pairwise inequivalent irreducible representations) if and only if the following conditions hold:
\begin{enumerate}[(i)]
\item $G=G_1\times G_2$ and the isotropy representations $\pg_1,\pg_2$ of $G_1/\pi_1(K)$ and $G_2/\pi_2(K)$, respectively, are both  multiplicity-free with pairwise inequivalent irreducible components. 

\item The center of $K$ has dimension $\leq 1$ (i.e., either $K$ is semisimple or $\dim{\kg_0}=1$).

\item None of the irreducible components of $\pg_1$ and $\pg_2$ is equivalent to any of the adjoint representations $\kg_0,\kg_1,\dots,\kg_t$.   
\end{enumerate} 

\begin{example}\label{dim5-E}
The lowest dimensional examples of this situation are the spaces $M^5=\SU(2)\times\SU(2)/S^1_{p,q}$, $p\ne q$ (see Example \ref{dim5}), which are the only cases with $\dim{K}=1$.  It is well known that these spaces all admit a unique invariant Einstein metric (see \cite[Example 6.9]{BhmWngZll}).  
\end{example}

In this section, we study multiplicity-free aligned homogeneous spaces $M=G/K=G_1\times G_2/K$ which in addition satisfy that 
$$
\mca^G=\{ g=(x_1,x_2,x_3):x_i>0\}, 
$$
so we need to assume that $G_1/\pi_1(K)$ and $G_2/\pi_2(K)$ are isotropy irreducible spaces and $K$ is simple.  This class of spaces will be called $\cca$.  

The existence of a $G$-invariant Einstein metric on a space in $\cca$ is therefore a case covered by Theorem \ref{s2E} and Proposition \ref{KssE}.  In this case, the graph is always connected and the B\"ohm's simplicial complex and Graev's nerve are both contractible (see \cite{BhmKrr2}).  Indeed, the only intermediate subalgebras are 
$$
\kg\oplus\pg_3 \subset \kg\oplus\pg_1\oplus\pg_3, \quad \kg\oplus\pg_1\oplus\pg_3.  
$$
Thus the existence of a $G_1\times G_2$-invariant Einstein metric on $M=G_1\times G_2/K$ does not follow from any known general existence theorem, it is actually equivalent by Proposition \ref{KssE} to the existence of a real root for the quartic polynomial $p$ given in \eqref{pE}, from which the following characterization follows.   

\begin{proposition}\label{mfE}
Let $M=G_1\times G_2/K$ be a homogeneous space in the class $\cca$, i.e., $G_1/\pi_1(K)$, $G_2/\pi_2(K)$ are different isotropy irreducible spaces and $K$ is simple, and set the numbers
$$
n_1:=\dim{G_1/K}, \quad n_2:=\dim{G_2/K}, \quad d:=\dim{K}, \quad 
\kil_\kg=a_1\kil_{\ggo_1}|_\kg, \quad \kil_\kg=a_2\kil_{\ggo_2}|_\kg.  
$$
Then $M$ admits a $G_1\times G_2$-invariant Einstein metric if and only if one of the following inequalities holds:  
\begin{enumerate}[{\rm (i)}]
\item $\Delta<0$. 

\item $\Delta>0$, $R<0$ and $S<0$.

\item $\Delta=0$, $S\leq 0$ or $T\ne 0$, 
\end{enumerate}
where $\Delta$, $R$, $S$ and $T$ are given in terms of $n_1,n_2,d,a_1,a_2$ as in Section \ref{Kss-sec}.  
\end{proposition}

\begin{remark}
The case when $K$ is either semisimple (non-simple) or has a one-dimensional center will be considered in \cite{Rical}.
\end{remark}

\begin{example}\label{SOd}
A general construction of homogeneous spaces in $\cca$ can be given using Example \ref{so} as follows: given any isotropy irreducible space $H/K$ with $K$ simple, consider $\SO(d)/K$, where $d=\dim{K}$, and the aligned homogeneous space $M^n=\SO(d)\times H/\Delta K$.  Thus $n=\tfrac{d(d-1)}{2}+n_2$, where $n_2=\dim{H}-d$, and if $\kil_{\kg}=a_2\kil_{\hg}|_{\kg}$, then 
$$
a_1=\tfrac{1}{d-2}, \quad c_1=\tfrac{(d-2)a_2+1}{(d-2)a_2}, \quad \lambda_1=\tfrac{a_2}{(d-2)a_2+1}, \quad \kappa_1=\tfrac{2}{d-2},  \quad \kappa_2=\tfrac{d(1-a_2)}{n_2}. 
$$
This subclass of $\cca$ consists of $7$ infinite families, where $H/K$ belongs to one of the families in lines 1,2,3 of \cite[Table 3]{HHK} and lines 2,3,5,6 of \cite[Table 4]{HHK}, and $24$ isolated spaces, where $H/K$ is one of the spaces with $K$ simple in \cite[Tables 3,6,7]{HHK}.  
\end{example}

The class $\cca$ can be classified using Table \ref{ii-simple}, which contains all the isotropy irreducible homogeneous spaces $G_i/K$ with $K$ simple and was obtained from  \cite[Tables 3,4,5,6,7,9]{HHK}.  A careful inspection of Table \ref{ii-simple} gives that the class $\cca$ consists of $12$ infinite families and $70$ sporadic examples.  The existence problem for invariant Einstein metrics among $\cca$ can be solved by computing the signs of the invariants $\Delta,R,S$ given in \S\ref{Kss-sec} with the help of Maple.  The results obtained are shown in three tables:  
\begin{enumerate}[{\small $\bullet$}]
\item Among the $12$ families, existence mostly holds, there are only $3$ non-existence infinite families (see Table \ref{flies}).   

\item All the spaces such that $G_1/\pi_1(K)$ and $G_2/\pi_2(K)$ are both irreducible symmetric spaces are listed in Table \ref{sym} ($1$ family and $6$ examples).  Non-existence prevails.  

\item A number of $24$ of the $70$ sporadic examples are given in Table \ref{spo}, among which existence holds for $16$ of them.  This table includes all the spaces with an exceptional $K$, as well as with the smallest $K$'s which do not belong to any infinite family: $\SU(2)$, $\SU(3)$, $\Gg_2$, $\Spe(3)$.   

\item The remaining $41$ sporadic examples are listed in Table \ref{spo2}.  Only $6$ of them do not admit an invariant Einstein metric.     
\end{enumerate}

Each sporadic case was worked out using Maple in two different ways: 1) by computing the signs of the invariants $\Delta,R,S$ given in \S\ref{Kss-sec}, and 2) by directly solving the equations \eqref{E6} and \eqref{E7} (or equivalently, \eqref{E3} and \eqref{E4}).   

The following observations on this classification are in order:
\begin{enumerate}[{\small $\bullet$}]
\item All the existence cases have $\Delta(p)<0$, so there are exactly two Einstein metrics on each space for which existence holds.  

\item The non-existence cases all have $\Delta(p), R(p)>0$ (cf.\ conditions above Example \ref{21}).

\item If both spaces $G_1\times G_1/K$ and $G_2\times G_2/K$ admit a diagonal Einstein metric, i.e., according to Remark \ref{hhk},
$$
(2\kappa_i+1)^2\geq 8a_i(1-a_i+\kappa_i), \qquad i=1,2, 
$$
then there is an Einstein metric on the aligned space $M=G_1\times G_2/K$.   

\item The converse to the above assertion does not hold.  
\end{enumerate}

We do not know whether the above properties can be prove without using the classification, there may be a conceptual reason behind them.

\begin{table} 
{\tiny 
$$
\begin{array}{c|c|c|c|c|c|c|c|c|c}
M^n=G/K & n & d & n_1 & n_2 & a_1 & a_2 & c_1 & \lambda  &  
\\[2mm] \hline \hline \rule{0pt}{14pt}
\Spe(2)\times\SU(3)/\SU(2) & 15 & 3 & 7 & 5 & \tfrac{1}{15} & \tfrac{1}{6} & \tfrac{7}{2} & \tfrac{1}{21} & \exists 
\\[2mm]  \hline \rule{0pt}{14pt}
\Gg_2\times\SU(3)/\SU(2) & 19 & 3 & 11 & 5 & \tfrac{1}{56} & \tfrac{1}{6} & \tfrac{31}{28} & \tfrac{1}{62} & \exists 
\\[2mm]  \hline \rule{0pt}{14pt}
\Gg_2\times\Spe(2)/\SU(2) & 21 & 3 & 11 & 7 & \tfrac{1}{56} & \tfrac{1}{15} & \tfrac{71}{56} & \tfrac{1}{71} & \exists 
\\[2mm]  \hline \rule{0pt}{14pt}
\SO(8)\times \Gg_2/\SU(3) & 34 & 8 & 20 & 6 & \tfrac{1}{6} & \tfrac{3}{4} & \tfrac{11}{9} & \tfrac{3}{22} & \nexists 
\\[2mm]  \hline \rule{0pt}{14pt}
\SU(6)\times \Gg_2/\SU(3) & 41 & 8 & 27 & 6 & \tfrac{1}{10} & \tfrac{3}{4} & \tfrac{17}{15} & \tfrac{3}{34} & \nexists 
\\[2mm]  \hline \rule{0pt}{14pt}
\Eg_6\times \Gg_2/\SU(3) & 84 & 8 & 70 & 6 & \tfrac{1}{36} & \tfrac{3}{4} & \tfrac{28}{27} & \tfrac{3}{112} & \nexists 
\\[2mm]  \hline \rule{0pt}{14pt}
\Eg_7\times \Gg_2/\SU(3) & 139 & 8 & 125 & 6 & \tfrac{1}{126} & \tfrac{3}{4} & \tfrac{191}{189} & \tfrac{3}{382} & \nexists 
\\[2mm]  \hline \rule{0pt}{14pt}
\SU(6)\times\SO(8)/\SU(3) & 55 & 8 & 27 & 20 & \tfrac{1}{10} & \tfrac{1}{6} & \tfrac{8}{5} & \tfrac{1}{16} & \exists 
\\[2mm]  \hline \rule{0pt}{14pt}
\Eg_6\times\SO(8)/\SU(3) & 98 & 8 & 70 & 20 & \tfrac{1}{36} & \tfrac{1}{6} & \tfrac{7}{6} & \tfrac{1}{42} & \exists 
\\[2mm]  \hline \rule{0pt}{14pt}
\Eg_7\times\SO(8)/\SU(3) & 153 & 8 & 125 & 20 & \tfrac{1}{126} & \tfrac{1}{6} & \tfrac{22}{21} & \tfrac{1}{132} & \exists 
\\[2mm]  \hline \rule{0pt}{14pt}
\Eg_6\times\SU(6)/\SU(3) & 105 & 8 & 70 & 27 & \tfrac{1}{36} & \tfrac{1}{10} & \tfrac{23}{18} & \tfrac{1}{46} & \exists 
\\[2mm]  \hline \rule{0pt}{14pt}
\Eg_7\times\SU(6)/\SU(3) & 160 & 8 & 125 & 27 & \tfrac{1}{126} & \tfrac{1}{10} & \tfrac{68}{63} & \tfrac{1}{136} & \exists 
\\[2mm]  \hline \rule{0pt}{14pt}
\Eg_7\times\Eg_6/\SU(3) & 203 & 8 & 125 & 70 & \tfrac{1}{126} & \tfrac{1}{36} & \tfrac{9}{7} & \tfrac{1}{162} & \exists 
\\[2mm]  \hline \rule{0pt}{14pt}
\Eg_6\times\SO(7)/\Gg_2 & 85 & 14 & 64 & 7 &\tfrac{1}{9} &\tfrac{4}{5}  & \tfrac{41}{36} & \tfrac{4}{41} & \nexists 
\\[2mm]  \hline \rule{0pt}{14pt}
\SO(14)\times\SO(7)/\Gg_2 & 98 & 14 & 77 & 7 &\tfrac{1}{12} & \tfrac{4}{5}  & \tfrac{53}{48} & \tfrac{4}{53} & \nexists 
\\[2mm]  \hline \rule{0pt}{14pt}
\SO(14)\times\Eg_6/\Gg_2 & 155 & 14 & 77 & 64 &\tfrac{1}{12} & \tfrac{1}{9}  & \tfrac{7}{4} & \tfrac{1}{21} & \exists 
\\[2mm]  \hline \rule{0pt}{14pt}
\Spe(7)\times\SO(14)/\Spe(3) & 175 & 21 & 84 & 70 &\tfrac{1}{10} &\tfrac{13}{18}  & \tfrac{74}{65} & \tfrac{13}{148} & \nexists 
\\[2mm]  \hline \rule{0pt}{14pt}
\Spe(7)\times\SU(6)/\Spe(3) & 119 & 21 & 84 & 14 &\tfrac{1}{10} & \tfrac{2}{3}  & \tfrac{23}{20} & \tfrac{2}{23} & \exists 
\\[2mm]  \hline \rule{0pt}{14pt}
\SO(21)\times\Spe(7)/\Spe(3) & 295 & 21 & 189 & 84 &\tfrac{1}{19} & \tfrac{1}{10}  & \tfrac{29}{19} & \tfrac{1}{29} & \exists 
\\[2mm]  \hline \rule{0pt}{14pt}
\SO(26)\times \Eg_6/\Fg_4 & 351 & 52  & 273 & 26 & \tfrac{1}{8} & \tfrac{3}{4} & \tfrac{7}{6} & \tfrac{3}{28} & \nexists
\\[2mm]  \hline \rule{0pt}{14pt}
\SO(52)\times\Eg_6/\Fg_4 & 1352 & 52  & 1274 & 26 & \tfrac{1}{50} & \tfrac{3}{4} & \tfrac{77}{75} & \tfrac{3}{154} & \exists
\\[2mm]  \hline \rule{0pt}{14pt}
\SO(52)\times\SO(26)/\Fg_4 & 1599 & 52  & 1274 & 273 & \tfrac{1}{50} & \tfrac{1}{8} & \tfrac{29}{25} & \tfrac{1}{58} & \exists
\\[2mm]  \hline \rule{0pt}{14pt}
\SO(78)\times\SU(27)/\Eg_6 & 3653 & 78 & 2925 &650 &\tfrac{1}{76} & \tfrac{2}{27} & \tfrac{179}{152} & \tfrac{2}{179} & \exists 
\\[2mm]  \hline \rule{0pt}{14pt}
\SO(133)\times\Spe(28)/\Eg_7 & 10291 & 133 & 8645 & 1463&\tfrac{1}{131} & \tfrac{3}{58} & \tfrac{451}{393} & \tfrac{3}{451} & \exists 
\\[2mm] \hline\hline
\end{array}
$$}
\caption{A list of $24$ {\bf sporadic} examples.  Here $\kappa_i=\tfrac{d(1-a_i)}{n_i}$,  $c_1=\tfrac{a_1+a_2}{a_2}$ and $\lambda=\tfrac{a_1a_2}{a_1+a_2}$.} 
\label{spo}
\end{table}

\begin{table} 
{\tiny 
$$
\begin{array}{c|c||c|c}
M=G_1\times G_2/K & & M=G_1\times G_2/K &     
\\[2mm] \hline \hline \rule{0pt}{14pt}
\Spe(10)_{175,\tfrac{1}{11}}\times\SU(15)_{189,\tfrac{1}{10}}/\SU(6)_{35} & \exists & \SU(28)_{720,\tfrac{1}{21}}\times\underline{\Eg_7}_{70,\tfrac{4}{9}}/\SU(8)_{63} & \exists 
\\[2mm] \hline \rule{0pt}{14pt}
\SU(21)_{405,\tfrac{1}{28}}\times\Spe(10)_{175,\tfrac{1}{11}}/\SU(6)_{35} & \exists & \SU(36)_{1232,\tfrac{1}{45}}\times\underline{\Eg_7}_{70,\tfrac{4}{9}}/\SU(8)_{63} & \exists 
\\[2mm] \hline \rule{0pt}{14pt}
\SO({\bf 35})_{560,\tfrac{1}{33}}\times\Spe(10)_{175,\tfrac{1}{11}}/\SU(6)_{35} & \exists & \SO({\bf 63})_{1890,\tfrac{1}{61}}\times\underline{\Eg_7}_{70,\tfrac{4}{9}}/\SU(8)_{63}  & \exists 
\\[2mm] \hline \rule{0pt}{14pt}
\underline{\SO(16)}_{84,\tfrac{1}{4}}\times\SO(10)_{9,\tfrac{7}{8}}/\SO(9)_{36}& \nexists & \underline{\SO(70)}_{2352,\tfrac{1}{85}}\times\underline{\Eg_7}_{70,\tfrac{4}{9}}/\SU(8)_{63} & \exists 
\\[2mm] \hline \rule{0pt}{14pt}
\underline{\SO(16)}_{84,\tfrac{1}{4}}\times\underline{\Fg_4}_{16,\tfrac{7}{9}}/\SO(9)_{36}& \nexists & \underline{\SO(70)}_{2352,\tfrac{1}{85}}\times\SU(28)_{720,\tfrac{1}{21}}/\SU(8)_{63} & \exists 
\\[2mm] \hline \rule{0pt}{14pt}
\SO({\bf 36})_{594,\tfrac{1}{34}}\times\underline{\Fg_4}_{16,\tfrac{7}{9}}/\SO(9)_{36}& \exists & \underline{\SO(70)}_{2352,\tfrac{1}{85}}\times\SU(36)_{1232,\tfrac{1}{45}}/\SU(8)_{63}  & \exists 
\\[2mm] \hline \rule{0pt}{14pt}
\SO(44)_{910,\tfrac{1}{66}}\times\underline{\Fg_4}_{16,\tfrac{7}{9}}/\SO(9)_{36}& \exists & \underline{\SO(70)}_{2352,\tfrac{1}{85}}\times\SO({\bf 63})_{1890,\tfrac{1}{61}}/\SU(8)_{63} & \exists 
\\[2mm] \hline \rule{0pt}{14pt}
\underline{\SO(16)}_{84,\tfrac{1}{4}}\times\SU(9)_{44,\tfrac{7}{18}}/\SO(9)_{36}& \exists & \underline{\Spe(16)}_{462,\tfrac{5}{68}}\times\SO(13)_{12,\tfrac{10}{11}}/\SO(12)_{66} & \nexists 
\\[2mm] \hline \rule{0pt}{14pt}
\SO({\bf 36})_{594,\tfrac{1}{34}}\times\underline{\SO(16)}_{84,\tfrac{1}{4}}/\SO(9)_{36} & \exists & \underline{\Spe(16)}_{462,\tfrac{5}{68}}\times\SU(12)_{77,\tfrac{5}{12}}/\SO(12)_{66} & \exists 
\\[2mm] \hline \rule{0pt}{14pt}
\SO(44)_{910,\tfrac{1}{66}}\times\underline{\SO(16)}_{84,\tfrac{1}{4}}/\SO(9)_{36} & \exists & \SO({\bf 66})_{2079,\tfrac{1}{64}}\times\underline{\Spe(16)}_{462,\tfrac{5}{68}}/\SO(12)_{66}  & \exists 
\\[2mm] \hline \rule{0pt}{14pt}
\underline{\SO(42)}_{825,\tfrac{1}{56}}\times\SU(8)_{27,\tfrac{5}{8}}/\Spe(4)_{36} & \exists & \SO(77)_{2860,\tfrac{1}{105}}\times\underline{\Spe(16)}_{462,\tfrac{5}{68}}/\SO(12)_{66} & \exists
\\[2mm] \hline \rule{0pt}{14pt}
\underline{\Eg_6}_{42,\tfrac{5}{12}}\times\SO(27)_{315,\tfrac{23}{30}}/\Spe(4)_{36} & \nexists & \SU(36)_{1215,\tfrac{1}{28}}\times\underline{\Eg_8}_{168,\tfrac{3}{10}}/\SU(9)_{80} & \exists
\\[2mm] \hline \rule{0pt}{14pt}
\SO({\bf 36})_{594,\tfrac{1}{34}}\times\underline{\Eg_6}_{42,\tfrac{5}{12}}/\Spe(4)_{36} & \exists & \SU(45)_{1944,\tfrac{1}{55}}\times\underline{\Eg_8}_{168,\tfrac{3}{10}}/\SU(9)_{80}& \exists
\\[2mm] \hline \rule{0pt}{14pt}
\underline{\SO(42)}_{825,\tfrac{1}{56}}\times\underline{\Eg_6}_{42,\tfrac{5}{12}}/\Spe(4)_{36} & \exists & \SO({\bf 80})_{3080,\tfrac{1}{78}}\times\underline{\Eg_8}_{168,\tfrac{3}{10}}/\SU(9)_{80} & \exists
\\[2mm] \hline \rule{0pt}{14pt}
\underline{\SO(42)}_{825,\tfrac{1}{56}}\times\SO(27)_{315,\tfrac{23}{30}}/\Spe(4)_{36} & \exists & \underline{\SO(128)}_{8008,\tfrac{1}{144}}\times\SO(17)_{16,\tfrac{14}{15}}/\SO(16)_{120} & \nexists
\\[2mm] \hline \rule{0pt}{14pt}
\underline{\SO(42)}_{825,\tfrac{1}{56}}\times\SO({\bf 36})_{594,\tfrac{1}{34}}/\Spe(4)_{36} & \exists & \SO({\bf 120})_{7020,\tfrac{1}{118}}\times\underline{\Eg_8}_{128,\tfrac{7}{15}}/\SO(16)_{120} & \exists
\\[2mm] \hline \rule{0pt}{14pt}
\underline{\SU(16)}_{210,\tfrac{1}{8}}\times\SO(11)_{10,\tfrac{8}{9}}/\SO(10)_{45} & \nexists &  \underline{\SO(128)}_{8008,\tfrac{1}{144}}\times\underline{\Eg_8}_{128,\tfrac{7}{15}}/\SO(16)_{120} & \exists
\\[2mm] \hline \rule{0pt}{14pt}
\underline{\SU(16)}_{210,\tfrac{1}{8}}\times\SU(10)_{54,\tfrac{2}{5}}/\SO(10)_{45} & \exists &  \SO(135)_{8925,\tfrac{1}{171}}\times\underline{\Eg_8}_{128,\tfrac{7}{15}}/\SO(16)_{120} & \exists
\\[2mm] \hline \rule{0pt}{14pt}
\SO({\bf 45})_{945,\tfrac{1}{43}}\times\underline{\SU(16)}_{210,\tfrac{1}{8}}/\SO(10)_{45} & \exists & \underline{\SO(128)}_{8008,\tfrac{1}{144}}\times\SU(16)_{135,\tfrac{7}{16}}/\SO(16)_{120} & \exists
\\[2mm] \hline \rule{0pt}{14pt}
\SO(54)_{1386,\tfrac{1}{78}}\times\underline{\SU(16)}_{210,\tfrac{1}{8}}/\SO(10)_{45} & \exists & \underline{\SO(128)}_{8008,\tfrac{1}{144}}\times\SO({\bf 120})_{7020,\tfrac{1}{118}}/\SO(16)_{120} & \exists
\\[2mm] \hline \rule{0pt}{14pt}
& & \SO(135)_{8925,\tfrac{1}{171}}\times\underline{\SO(128)}_{8008,\tfrac{1}{144}}/\SO(16)_{120} & \exists
\\[2mm] \hline\hline
\end{array}
$$}
\caption{Remaining $41$ {\bf sporadic} examples.}\label{spo2}
\end{table}

\end{document}